\documentclass[12pt]{amsart}
\usepackage{amsmath}
\usepackage{graphicx}
\usepackage{amsfonts}
\usepackage{amsthm}
\usepackage{amssymb}
\usepackage{color}
\usepackage{slashed}
\usepackage{amscd}
\usepackage{mathtools}
\usepackage{hyperref}

\newtheorem{Thm}{Theorem}[section]
\newtheorem{Cor}[Thm]{Corollary}
\newtheorem{Defn}[Thm]{Definition}
\newtheorem{Lem}[Thm]{Lemma}
\newtheorem{Prop}[Thm]{Proposition}
\newtheorem{Eg}[Thm]{Example}
\newtheorem{Rmk}[Thm]{Remark}
\newenvironment{narrow}[2]{%
 \begin{list}{}{%
  \setlength{\topsep}{0pt}%
  \setlength{\leftmargin}{#1}%
  \setlength{\rightmargin}{#2}%
  \setlength{\listparindent}{\parindent}%
  \setlength{\itemindent}{\parindent}%
  \setlength{\parsep}{\parskip}%
 }%
\item[]}{\end{list}}

\newif\ifpic
\picfalse    
\pictrue   

\title{Embedding Seifert manifolds in $S^4$}
\author{Andrew Donald}
\email{a.donald.1@research.gla.ac.uk}

\begin{document}
\begin{abstract}
Using an obstruction based on Donaldson's theorem on the intersection forms of definite 4-manifolds, we determine which connected sums of lens spaces smoothly embed in $S^4$. We also find constraints on the Seifert invariants of Seifert 3-manifolds which embed in $S^4$ when either the base orbifold is non-orientable or the first Betti number is odd. In addition we construct some new embeddings and use these, along with the $d$ and $\overline{\mu}$ invariants, to examine the question of when the double branched cover of a 3 or 4 strand pretzel link embeds.
\end{abstract}
\maketitle
\section{Introduction}

We consider the question of which closed 3-manifolds can be embedded smoothly in $S^4$. Such manifolds are necessarily orientable. Results are known for special classes of manifolds including some Seifert fibred cases \cite{GL}, \cite{CH}, some of which also hold for topological locally flat embeddings. In the case of smooth embeddings, the question was examined systematically in \cite{Budney}.

The approach of this paper is partly motivated by work on knot theory. Recent work on slice knots, most notably by Lisca \cite{lisca} \cite{lisca2}, has focussed on obstructions to a rational homology sphere bounding a rational ball. If a knot is slice -- the boundary of a properly embedded 2-disk in $D^4$ -- then it is a classical fact that its double branched cover bounds a rational ball. We adopt a similar approach to 3-manifolds embedding in $S^4$ using the following observation. If a knot (or indeed a link) is doubly slice -- that is, a cross-section of an unknotted 2-sphere in $S^4$ -- its double branched cover embeds in $S^4$. Since a doubly slice knot is automatically slice it is natural to use obstructions of a similar flavour.

Lisca's work on 3-dimensional lens spaces and two-bridge links determined precisely which connected sums of lens spaces were the boundaries of smooth rational balls. The same methods can be adapted to determine which embed smoothly in $S^4$. Recall that each lens space can be written as $L(p,q)$ with $p>q>0$ and is given by $-p/{q}$-surgery on the unknot in $S^3$.

\begin{Thm} \label{sumslens}
Let $L= \#_{i=1}^h{L(p_i , q_i)}$. Then $L$ embeds smoothly in $S^4$ if and only if each $p_i$ is odd and there exists $Y$ such that $L \cong Y \# -Y$.
\end{Thm}

This generalises a result of Gilmer-Livingston \cite{GL} and Fintushel-Stern \cite{fintushel-sternssf} in the case $h=2$.

The primary obstruction here utilises Donaldson's diagonalisation theorem and we briefly summarise it. Since a connected sum of lens spaces is a rational homology sphere, an embedding into $S^4$ produces a splitting $S^4=U \cup_L -V$ where $U$ and $V$ are rational balls with common boundary $L$. For either orientation of $L$ there is a standard negative definite 2-handlebody\footnote{We use the term `2-handlebody' to refer to a 4-manifold produced by attaching 2-handles to $D^4$.} with boundary $L$, given by a plumbing construction. (See \cite[Example 4.6.2]{GS} for details on plumbings.)

\begin{figure}[htbp] 
\begin{center}
\ifpic
\def\svgwidth{5cm}
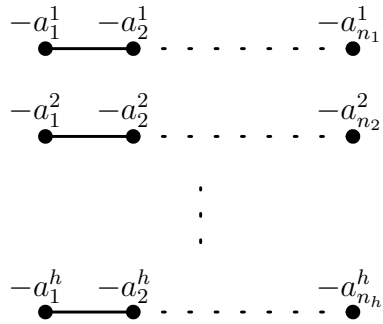
\else \vskip 5cm \fi
\begin{narrow}{0.3in}{0.3in}
\caption{
\bf{Plumbing graph for a negative definite 4-manifold with boundary a connected sum of lens spaces.}}
\label{h linear plumbings}
\end{narrow}
\end{center}
\end{figure}

In brief, we find $a_1^i, \ldots , a_{n_i}^i$ such that $$\frac{p_i}{q_i}  = [a_1^i, \ldots , a_{n_i}^i]^- = a_1^i -\dfrac{1}{a_2^i-\dfrac{1}{\ddots-\dfrac{1}{a_{n_i}^i}}} $$ with each $a_j^i \geq 2$. This is the negative continued fraction of ${p_i}/{q_i}$. Let $X$ be obtained by plumbing according to the graph in Figure \ref{h linear plumbings}. This has boundary $L$ and is negative definite so gluing $-U$ or $-V$ to $X$ gives a smooth, closed, negative definite 4-manifold. This must have a standard intersection form. Lisca studied the induced map $H_2(X) \to H_2(X \cup -U)$ to obtain conditions on the intersection form of $X$ and hence $L$. We get an obstruction to $L$ embedding in $S^4$ by using both $U$ and $V$.

These ideas can be extended to Seifert manifolds \cite{Lec}, \cite{GJ}.
A Seifert manifold $Y$ can be described by a base surface $F$ together with a collection of singular fibres described by Seifert invariants of type $(a_i,b_i)$ for coprime integers $a_i,b_i$. Such a description is not unique but it is not hard to determine when two sets of data give the same Seifert manifold \cite{NR}.

\begin{figure}[htbp] 
\begin{center}
\ifpic
\def\svgwidth{8cm}
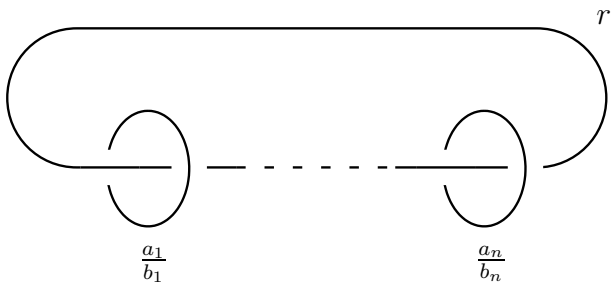
\else \vskip 10cm \fi
\begin{narrow}{0.3in}{0.3in}
\caption{
\bf{$Y(S^2;r;(a_1,b_1), \ldots , (a_n,b_n))$.}}
\label{genseif}
\end{narrow}
\end{center}
\end{figure}

It will be convenient to arrange that every singular fibre has $a_i>1$. We will therefore describe a Seifert manifold as $$Y(F;r;(a_1,b_1), \ldots ,(a_n,b_n)),$$
where each $a_i>1$ with $a_i,b_i$ coprime and $r \in \mathbb{Z}$. A surgery diagram is given in Figure \ref{genseif} when $F=S^2$. Note that this differs slightly to the notation of \cite{CH} and \cite{NR} where the Seifert invariants are always chosen so that $r=0$. We will call the surgery curve with framing $r$ the central curve. The generalised Euler invariant of $Y$ is given by $$e(Y) = \frac{b_1}{a_1}+ \ldots + \frac{b_n}{a_n} -r.$$

When $Y$ is a Seifert manifold with base orbifold $S^2$ and $e(Y)>0$ a negative definite 4-manifold with boundary $Y$ can be obtained by a standard plumbing construction \cite{NR}. The same construction gives a semi-definite 4-manifold when $e=0$. Figure \ref{allp3}(c) shows this plumbing for $Y(S^2;0;(3,1),(3,-1),(3,1))$.

\begin{figure}[hbp] 
\begin{center}
\ifpic
\def\svgwidth{10cm}
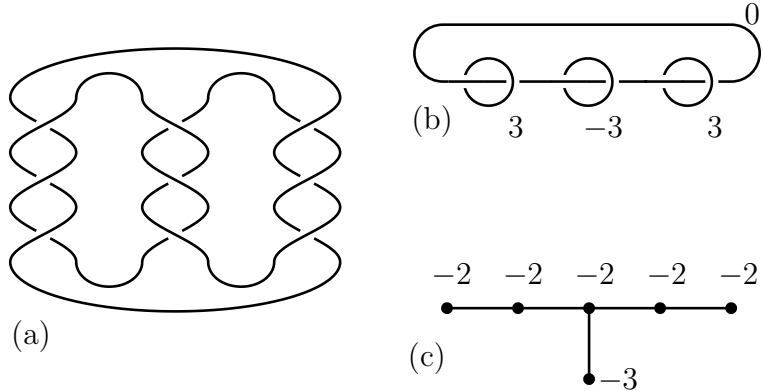
\else \vskip 5cm \fi
\begin{narrow}{0.3in}{0.3in}
\caption{
\bf{(a) The pretzel knot $P(3,-3,3)$; (b) a surgery diagram for its double branched cover $Y(3,-3,3)$; (c) the plumbing graph for a negative definite 4-manifold with boundary $Y(3,-3,3)$.}}
\label{allp3}
\end{narrow}
\end{center}
\end{figure}
With a minor modification, we can find negative definite 4-manifolds bounding Seifert manifolds with any base orbifold. We describe these 4-manifolds in Proposition \ref{definites}. In particular, for a non-orientable base surface, we obtain a negative definite 4-manifold regardless of $e(Y)$.

We get a result similar to Theorem \ref{sumslens}. A pair of Seifert invariants are called complementary if they are equivalent to ones of the form $(a,b), (a, -b)$. We also extend this to a notion of weak complementary pairs by allowing pairs $(a,b), (a,-b')$ where $bb' \equiv 1$ mod $a$.


\begin{Thm}\label{non-ori}
Let $Y$ be a Seifert manifold with non-orientable base surface $F$. If $Y$ embeds smoothly in $S^4$ then the Seifert invariants of $Y$ occur in weak complementary pairs. In addition, whenever there are Seifert invariants $(a_i,b_i), (a_j,b_j)$ with $a_i,a_j$ both even, then $a_i=a_j$ and $b_i \in \{\pm b_j, \pm b_j'\}$.
\end{Thm}

While this result does not put any restriction on the Euler invariant of $Y$, it is shown in \cite{CH} that for a given set of Seifert invariants there are only finitely many possible values of $e(Y)$ for which an embedding is possible and, in the case of complementary pairs with every $a_i$ odd, these are completely described.

We also consider orientable base surfaces. An interesting special case, considered by Hillman \cite{hillman}, occurs when $e(Y)=0$. These are the only examples where $b_1(Y)$ is odd.

\begin{Thm}\label{e=0}
Let $Y$ be a Seifert manifold with orientable base surface $F$ and $e(Y)=0$. If $Y$ embeds smoothly in $S^4$ then the Seifert invariants of $Y$ occur in complementary pairs.
\end{Thm}
\begin{Rmk}This holds even for topological embeddings when $F=S^2$ \cite{hillman}.
\end{Rmk}
When $Y$ has complementary pairs of Seifert invariants with every $a_i$ odd and $e(Y)=0$ it embeds smoothly in $S^4$ \cite{CH}.

The question of embedding for Seifert manifolds with orientable base and $e \neq 0$ appears to be more difficult so we will only consider the following special case.
Let $Y$ be a Seifert manifold with base surface $S^2$, at most 4 singular fibres, each described by $(a_i,b_i)$ with $b_i=\pm1$ and $r=0$. In this case, the legs in the standard negative definite plumbings will have a simpler form. Every leg will either consist of single vertex with a negative weight or a chain of vertices, all with weight $-2$. We will denote these manifolds as $Y(a_1b_1, \ldots , a_n b_n)$.
We will also assume $n \geq 3$ as this gives a lens space when $ n \leq 2$. Integer surgery diagrams are shown in Figure \ref{pretzelc}.

These manifolds are the double branched covers of pretzel links with up to 4 strands. The manifold $Y(a_1, \ldots , a_n)$ is the double branched cover of the pretzel link $P(a_1, \ldots , a_n)$.

\begin{figure}[htbp] 
\begin{center}
\ifpic
\def\svgwidth{5cm}
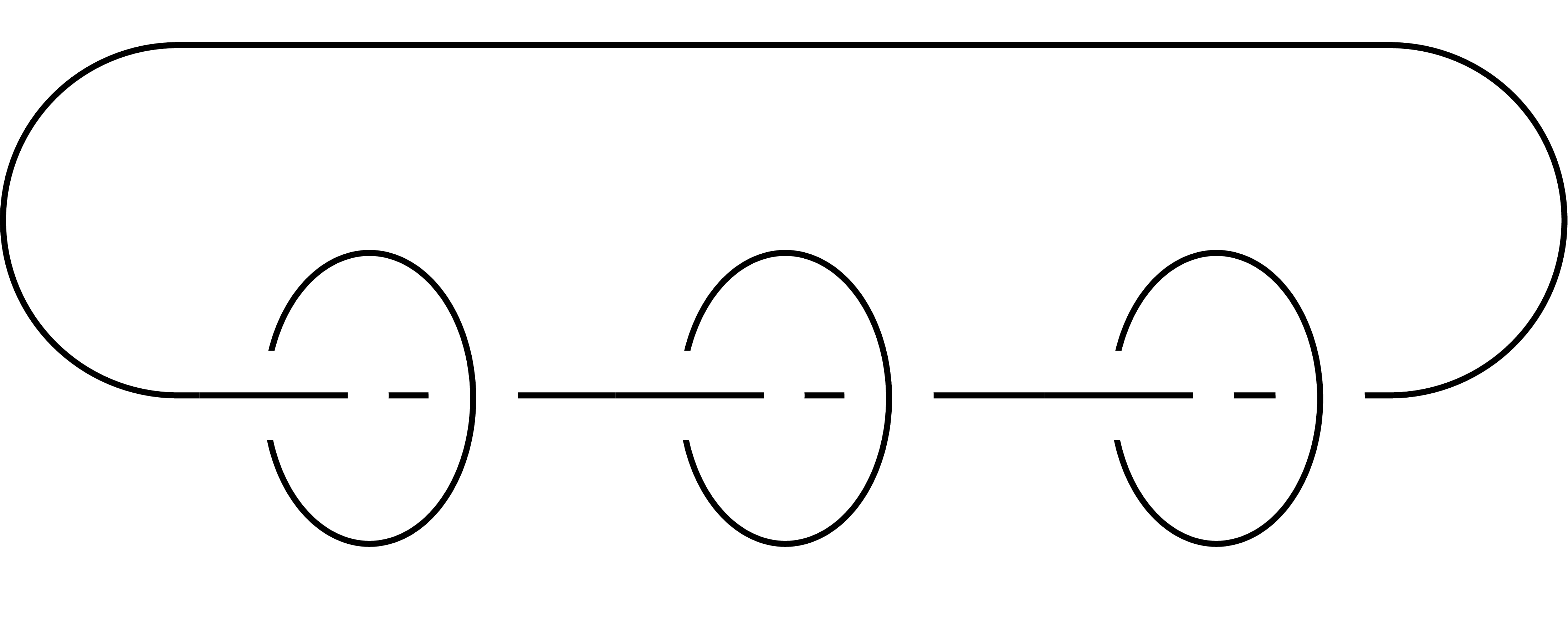
\vskip 1cm
\def\svgwidth{6cm}
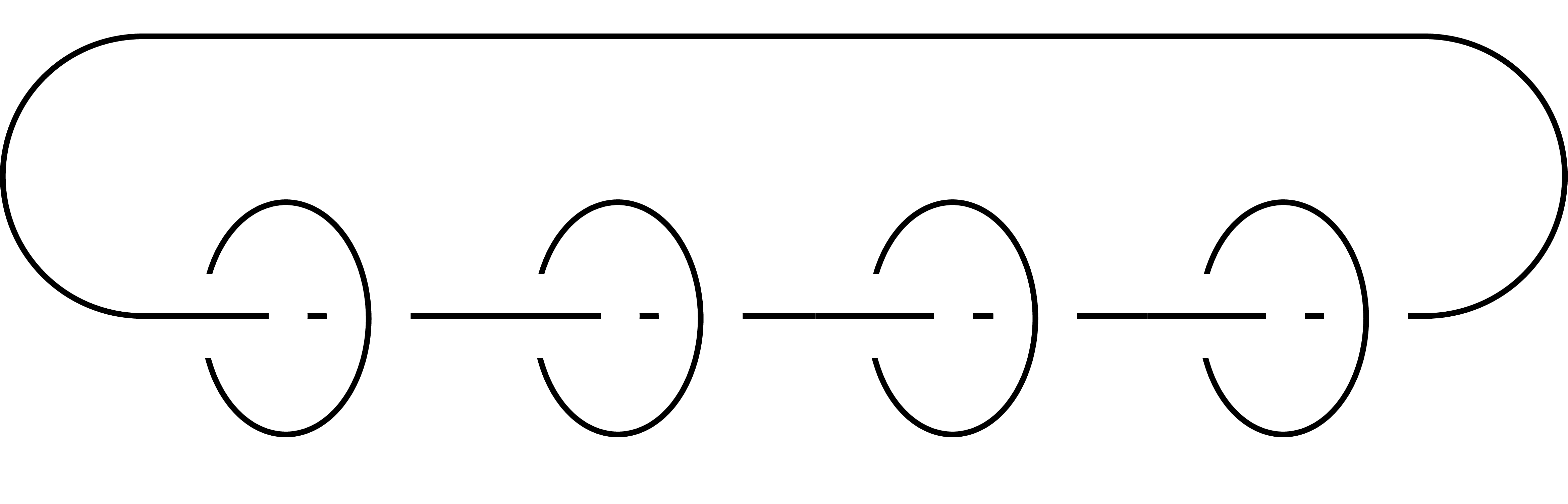
\else \vskip 5cm \fi
\begin{narrow}{0.3in}{0.3in}
\caption{
\bf{$Y(a,b,c)$ and $Y(a,b,c,d)$.}}
\label{pretzelc}
\end{narrow}
\end{center}
\end{figure}

With this notation in place, we have the following result.
\begin{Thm}\label{pretzelcover34}
Let $Y$ be of the form $Y(a,b,c)$ or $Y(a,b,c,d)$ where $a,b,c \in \mathbb{Z}\backslash\{-1, 0, 1\}$ and $d \in \mathbb{Z}\backslash \{0\}$.
If $Y$ embeds smoothly in $S^4$ then it is (possibly orientation-reversing) diffeomorphic to one of the following \begin{itemize} \item $Y(a,-a,a)$; 

\item $Y(a,-a,a,-a)$; \item $Y(a,-a,b,-b)$ with $b$ odd; \item $Y(a\pm 1,-a,a,-a)$;
\item$Y(2\lambda-1, -2\lambda-1, -2\lambda^2)$.
\end{itemize}

In addition, all but the last of these do embed smoothly in $S^4$.
\end{Thm}

Note that $Y(a,b,\pm1)$ and $Y(a,b,\pm1,\pm1)$ are lens spaces so the constraints imposed in the above statement are merely for convenience. 

This result relies on a combination of the obstruction from Donaldson's theorem and other methods. Specifically, we use the $d$ invariant of a spin$^c$ structure from Heegaard-Floer homology and the Neumann-Siebenmann $\overline{\mu}$ invariant of a spin structure, a lift of the Rochlin invariant. The former, as shown in \cite{GJ}, gives a useful strengthening of the obstruction from Donaldson's theorem. The embeddings arise as double branched covers of doubly slice links.

Some of the manifolds considered by Theorem \ref{pretzelcover34} have $e(Y)=0$. In particular, we see that the converse of Theorem \ref{e=0} is not true in general.

Embeddings can also be constructed using Kirby calculus. We use this to give the following example.
\begin{Eg}The manifold $Y(S^2;0;(4,1),(4,1),(12,-7))$ can be embedded smoothly in $S^4$.
\end{Eg}This manifold is $(1)$ from the list of unknown cases in Section 6 of \cite{Budney}.

\bf{Organisation}\normalfont.
Section \ref{construction} considers doubly slice links and produces various embeddings, in particular those required by Theorem \ref{pretzelcover34}. In Section \ref{diagonalisation}, the obstruction to embedding from Donaldson's diagonalisation theorem is derived and in Section \ref{lins} it is applied, in conjunction with the combinatorial machinery developed by Lisca, to prove Theorems \ref{sumslens}, \ref{non-ori} and \ref{e=0}. The relevant properties of the $d$ and $\overline{\mu}$ invariants are described in Section \ref{spin}. Finally in Section \ref{pretzelcovers}, Theorem \ref{pretzelcover34} is proved using the obstructions from Sections \ref{diagonalisation} and \ref{spin}.

\bf{Acknowledgements}\normalfont. I am grateful to Brendan Owens for suggesting this problem and many useful conversations. I also thank Ana Lecuona, Nikolai Saveliev and Jonathan Hillman for helpful comments. 

\section{Constructions}\label{construction}

\subsection{Constructing embeddings using doubly slice links}
This section will describe how to use doubly slice links to produce smooth embeddings of 3-manifolds in $S^4$.

An embedding of $S^n$ in $S^{n+2}$ is unknotted if it is the boundary of an embedded $D^{n+1}$. We will call a link $L$ in $S^3$ (smoothly) doubly slice if it is a cross-section of an unknotted (smooth) embedding of $S^2$ in $S^4$.

\begin{Lem}\label{dsemb}Let $L$ be a link in $S^3$ and $Y$ be the n-fold cyclic branched cover of $S^3$ with branch set $L$. If $L$ is smoothly doubly slice then $Y$ smoothly embeds in $S^4$.
\end{Lem}
\begin{proof}
The n-fold cyclic branched cover of $S^4$ with branch set an unknotted $S^2$ is $S^4$. This comes from repeated suspension of the unbranched n-fold cover of $S^1$ over itself, where the branched covering map is extended in the obvious way (see \cite[Example 10.B.4]{rolfsen}).

If $L$ is doubly slice then the pair $(S^3,L)$ sits inside $(S^4,S^2)$. The preimage of this subset gives $Y$ embedded in $S^4$.
\end{proof}

A source of doubly slice knots is Zeeman's twist-spinning construction \cite{Zeeman}:
\begin{Thm}\label{Zeemanthm}
Let $K$ be any knot. Then $K\#-K$ is doubly slice.
\end{Thm}
The special case when $K$ is a 2-bridge knot is of particular interest. The double branched cover of a 2-bridge knot is a lens space $L(p,q)$ with $p$ odd. All such lens spaces arise in this way so applying Zeeman's result to connected sums of 2-bridge knots gives all the embeddings required by Theorem \ref{sumslens}.

To produce more examples of doubly slice links we look at embeddings of spheres into $S^4$.

Let $f:S^2 \rightarrow S^4$ be a smooth embedding of a sphere $S$. We may delete a point in $S^4$ away from $S$. Then let $r:\mathbb{R}^4 \to \mathbb{R}$ be a projection such that $r \circ f$ is a Morse function for $S$. The preimage of each $t \in \mathbb{R}$ describes a link in $\mathbb{R}^3$, which will be denoted $S_t$, except at the isolated critical values of $r \circ f$. The isotopy type of these links only change when we pass through one of these critical values.
At a minimum or maximum of the Morse function the link changes by the addition or removal of an unknotted component while at a saddle point the cross-section changes by a band move.


We will use the following theorem of Scharlemann:
\begin{Thm}[Main theorem of \cite{Scharlemann}]
Let $\gamma_1$ and $\gamma_2$ be knots such that some band move on the split link $L=\gamma_1 \cup \gamma_2$ gives the unknot.
Then $\gamma_1$ and $\gamma_2$ are unknots and the band move is the connected sum.
\end{Thm}


From this, the following result can be obtained.\footnote{A similar statement appears in \cite{hosokawa}. The proof contains a gap which is repaired by \cite{Scharlemann}.}
\begin{Prop}
Let $S$ be a sphere in $S^4$. Suppose there is a projection $r$ so that the level sets of $S$ are such that $S_0$ is an unknot; all of the maxima occur at some level $t>0$; all of the minima occur at levels $t<0$ and every cross-section is a completely split link.

Then $S$ is an unknotted sphere.
\end{Prop}
Note that, by Scharlemann's result, all of the level sets are unlinks and at every saddle point the number of components increases as $|t|$ increases. 

\begin{proof}
The proof is by induction on the number of saddle points, $n$. The case $n=1$ follows easily from Scharlemann's result -- we may assume the sphere has two minima and one maximum and so the band move is just the connected sum of a pair of unknots. This describes an unknotted sphere.

Suppose $S$ has $n$ saddle points. It can be arranged that they occur at distinct levels. Let $t_n$ be the level of the top one. In order to increase the number of components, the band move at $t_n$ will just affect one of the components, $K$. By an isotopy, it can be arranged that the maxima capping off all the other components of the unlink here occur at level $t'<t_n$.

Choose some $t$ such that $t'<t<t_n$. The cross-section $S_t$ gives an unknot so there is a 2-disk ${D}$ at this level. Surgery along $D$ gives spheres $S'$ and $S''$. The Morse function of $S$ induces Morse function on these sphere with $2$ and $n-1$ saddle points respectively.
By induction, both are unknotted so bound 3-cells $D'$ and $D''$ respectively. These give $\overline{D}=D' \cup_{{D}} D''$, a 3-cell bound by $S$.
\end{proof}

\begin{Cor}\label{dscor}
Suppose $L$ is a link in $S^3$ and there are two sets of band moves $\{A_i\}_{1\leq i\leq k}$ and $\{B_j\}_{1\leq j\leq l}$ such that performing the moves
\begin{itemize}
\item $\{A_i\}_{1\leq i\leq k} \cup \{B_j\}_{1\leq j\leq l}$ gives an unknot;
\item $\{A_i\}_{1\leq i\leq k} \cup \{B_j\}_{j \leq n}$ gives an ${l-n+1}$-component unlink $(0 \leq n \leq l)$ and
\item$\{A_i\}_{i\leq n} \cup \{B_j\}_{1\leq j< l}$ gives an ${k-n+1}$-component unlink $(0 \leq n \leq k)$.
\end{itemize}

Then $L$ is doubly slice. In addition, a link obtained by performing any subset of this entire collection of band moves is doubly slice.
\end{Cor}

%

\begin{proof}
The above proposition can be applied to show that these band moves describe an unknotted sphere. Take the unknot obtained by using all of the bands as the central level set and undo the $A$ bands in order above it to get unlinks in the level sets above. Doing the same with the $B$ bands below gives an unknotted sphere. Changing the order of the band moves simply takes a different cross-section of the same sphere so the result follows.
\end{proof}


We will use this result to produce families of doubly slice links.

\begin{Prop}\label{lan}
Let $L_{a,n}$ be the link in Figure \ref{pan}. It is doubly slice for any $a,n \in \mathbb{Z}$.

\end{Prop}
\begin{figure}[htbp] 
\begin{center}
\ifpic
\small
\def\svgwidth{10cm}
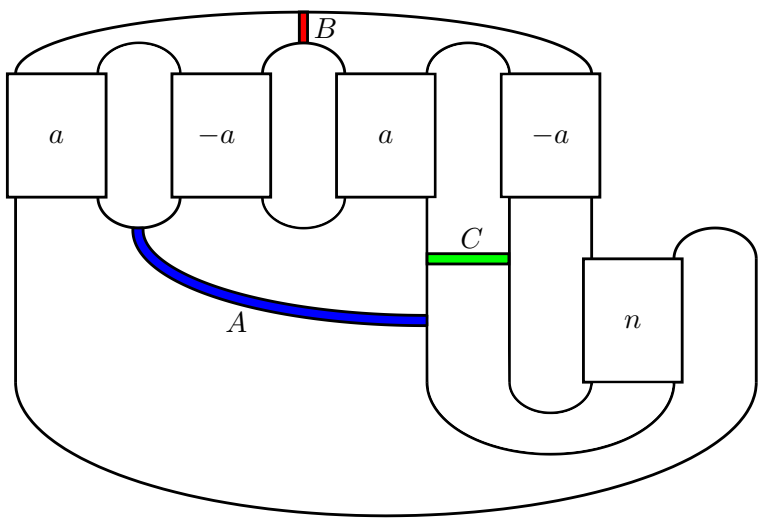
\else \vskip 5cm \fi
\begin{narrow}{0.3in}{0.3in}
\caption{
\bf{Band moves on $L_{a,n}$.}}
\label{pan}
\end{narrow}
\end{center}
\end{figure}

\begin{proof}
We ignore band $C$ for the moment.

After performing band move $A$, the crossings in $-a$ and $a$ twists in the second and third strands can be cancelled in pairs. The first and fourth strands may then also be removed so this gives a 2-component unlink.

Band move $B$ has a similar effect and also gives a 2-component unlink. Applying both band moves gives an unknot so we may apply Corollary \ref{dscor}.

\end{proof}

\begin{Cor}The pretzel links $P(a,-a,a)$, $P(a,-a,a,-a\pm1)$ and $P(a,-a,a,-a)$ are all doubly slice for any $a \in\mathbb{Z}$.
\end{Cor}
\begin{proof}The first two of these families are of the form $L_{a,n}$ when $n=0$ or $\pm1$. The unknotted sphere in Proposition \ref{lan} can be extended using band $C$. If we do band moves $B$ and $C$ we get a 3-component unlink so the three bands describe an unknotted sphere with three saddle points. The link given by band move $C$, $P(a,-a,a,-a)$, is therefore also doubly slice.
\end{proof}
\begin{Rmk}The manifolds $(2)-(8)$ listed in Section 6 and $(37)$ in Section 5\footnote{The $\mu$ invariants of this example are miscalculated in \cite{Budney}.} of \cite{Budney} are the double branched covers of links $L_{a,n}$ for small $a$ and $n$ and therefore embed smoothly in $S^4$.
\end{Rmk}

To construct more doubly slice pretzel links, we require the following intermediate result.

\begin{Lem}\label{torus}
Let $K$ be a $(2,2k+1)$-torus knot $T_{2,2k+1}$ for $k \geq 1$. Then $K \# -K$ is a cross-section of the unknotted sphere shown in Figure \ref{tor}, where the $2k$ bands are labelled as in Corollary \ref{dscor}.

A similar picture (with two bands) shows the same fact in the trivial case of $T_{2,1} \# -T_{2,1}$.

\begin{figure}[htbp] 
\begin{center}
\ifpic
\small
\def\svgwidth{10cm}
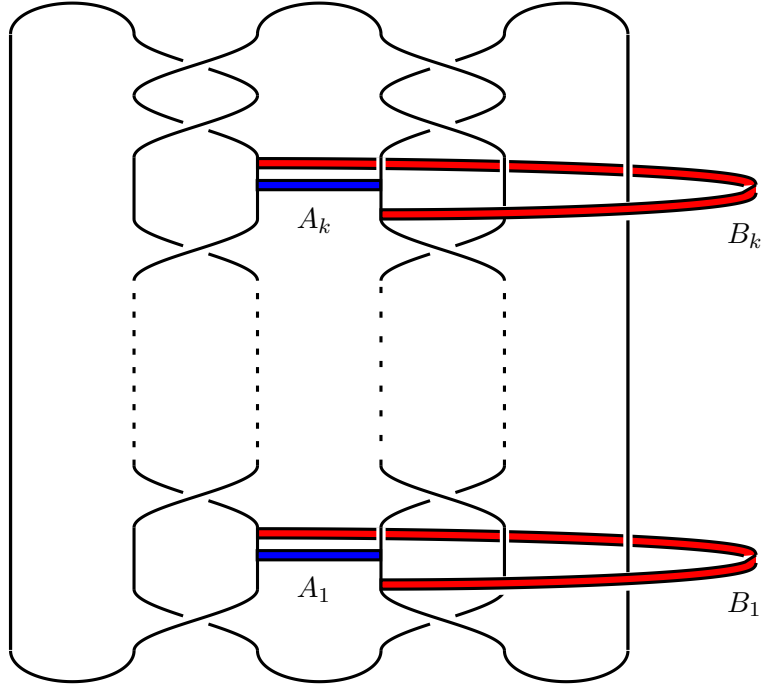
\else \vskip 5cm \fi
\begin{narrow}{0.3in}{0.3in}
\caption{
\bf{Band moves on $T_{2,2k+1} \# -T_{2,2k+1}$.}}
\label{tor}
\end{narrow}
\end{center}
\end{figure}

\end{Lem}

\begin{proof}We must verify that the bands in this picture satisfy the conditions of Corollary \ref{dscor}.
First, we claim that performing band moves $A_1$ and $B_1$ changes the sign of the crossing immediately above the pair of bands. The effect of band $A_1$ is shown in Figure \ref{toriso1} and there is an isotopy giving Figure \ref{toriso2}. Band move $B_1$ gives two pairs of canceling crossings and so transforms the knot to $T_{2,2k-1} \# -T_{2,2k-1}$. The rest of the bands are unaffected so we may continue this process with $k$ such pairs of band moves to produce the unknot.

\begin{figure}[htbp] 
\begin{center}
\ifpic
\small
\def\svgwidth{10cm}
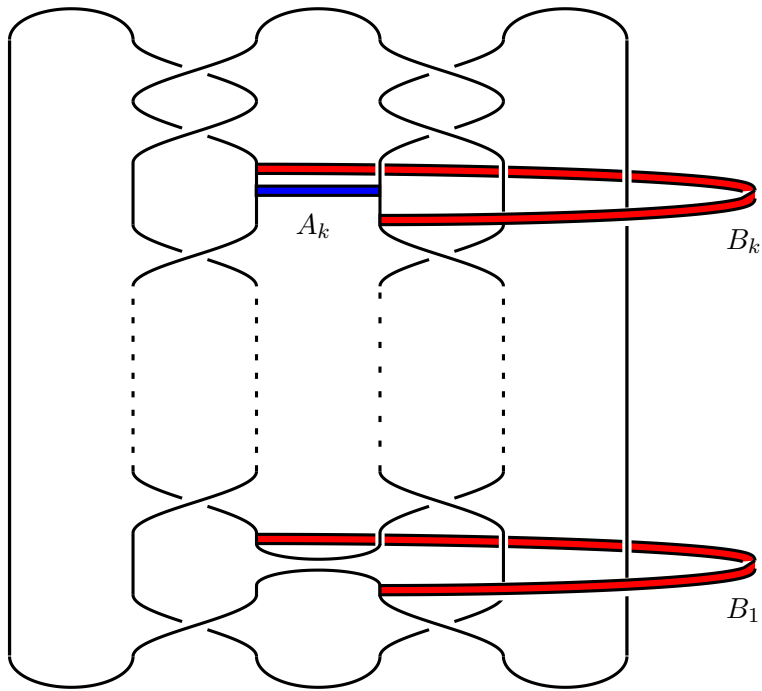
\else \vskip 5cm \fi
\begin{narrow}{0.3in}{0.3in}
\caption{
\bf{The result of band move $A_1$.}}
\label{toriso1}
\end{narrow}
\end{center}
\end{figure}

\begin{figure}[htbp] 
\begin{center}
\ifpic
\small
\def\svgwidth{10cm}
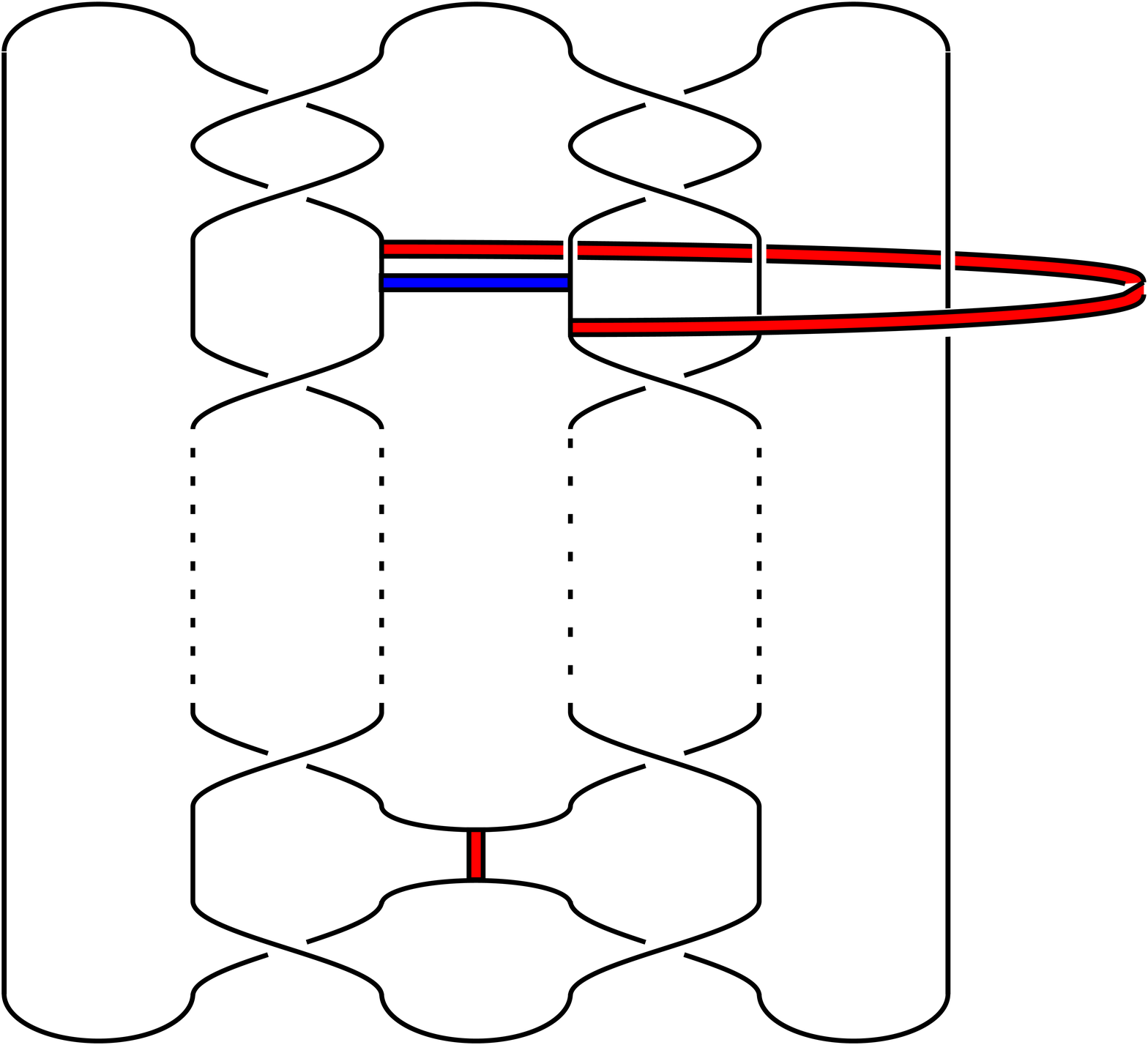
\else \vskip 5cm \fi
\begin{narrow}{0.3in}{0.3in}
\caption{
\bf{Isotopy simplifying band $B_1$ .}}
\label{toriso2}
\end{narrow}
\end{center}
\end{figure}

Now suppose we do all of the $A$ band moves and $B_1, \ldots B_n$ for some $n < k$. We begin by noting that when $i \leq n$ each pair $(A_i,B_i)$ cuts down the number of crossings, as before. It is therefore enough to show that applying the $k$ band moves $A_1, \ldots A_k$ to the diagram for $T_{2,2k+1} \# -T_{2,2k+1}$ gives a $k+1$ component unlink.

The band move $A_1$ gives a 2-component unlink as can  be seen in Figure \ref{toriso1} -- all of the crossings can be cancelled in pairs. Immediately after performing each subsequent band move, a further unlinked component can be removed.

There is an isotopy of Figure \ref{tor} which moves each band $B_i$ into the position that $A_i$ is drawn in. This can be seen by rotating the second factor in the connected sum anti-clockwise by $2 \pi$ through an axis passing through the band of the connected sum. This symmetry establishes that the above argument also works with each $A_i$ replaced by $B_i$, and so verifies the remaining condition in Corollary \ref{dscor}.
\end{proof}

We now show that $P(a,-a,b,-b)$ is doubly slice when $b$ is odd. There are two cases to consider.

\begin{Prop}\label{oddeven}
The link $P(a,-a,b,-b)$ is doubly slice when $a$ is even and $b$ is odd.
\end{Prop}
\begin{proof}
Figure \ref{oddeven1} shows that there is a band move using a band $C$ on $P(a,-a,b,-b)$ which gives $T_{2,|b-a|} \# -T_{2,|b-a|}$. Since $b-a$ is odd, Lemma \ref{torus} gives band moves on this knot satisfying Corollary \ref{dscor}. We can extend this picture by adding the band move $C$ and interpreting it as $B_{0}$.

\begin{figure}[htbp] 
\begin{center}
\ifpic
\small
\def\svgwidth{10cm}
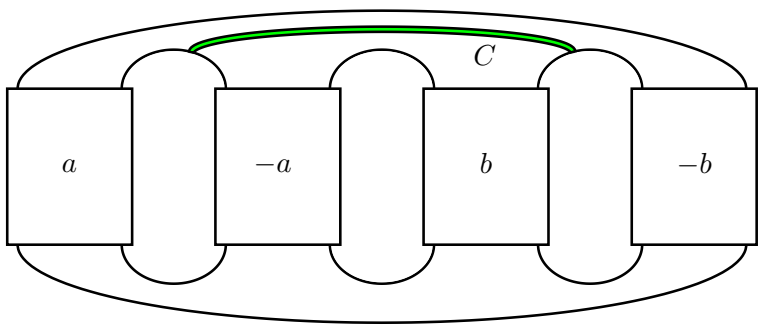
\else \vskip 5cm \fi
\begin{narrow}{0.3in}{0.3in}
\caption{
\bf{A band move on $P(a,-a,b,-b)$ .}}
\label{oddeven1}
\end{narrow}
\end{center}
\end{figure}
We claim that this picture also satisfies the conditions of Corollary \ref{dscor}. All but one of the cross-sections which need to be considered are obtained by applying a set of band moves including $C$ and so are described by Lemma \ref{torus}. Therefore the only thing that remains to be checked is that applying all of the band moves $A_i$ without $C$ gives an unlink with one more component that the one obtained by including $C$.

This is exhibited by Figure \ref{oddeven2}, with $2k+1 = |b-a|$.

\begin{figure}[htbp] 
\begin{center}
\ifpic
\small
\def\svgwidth{10cm}
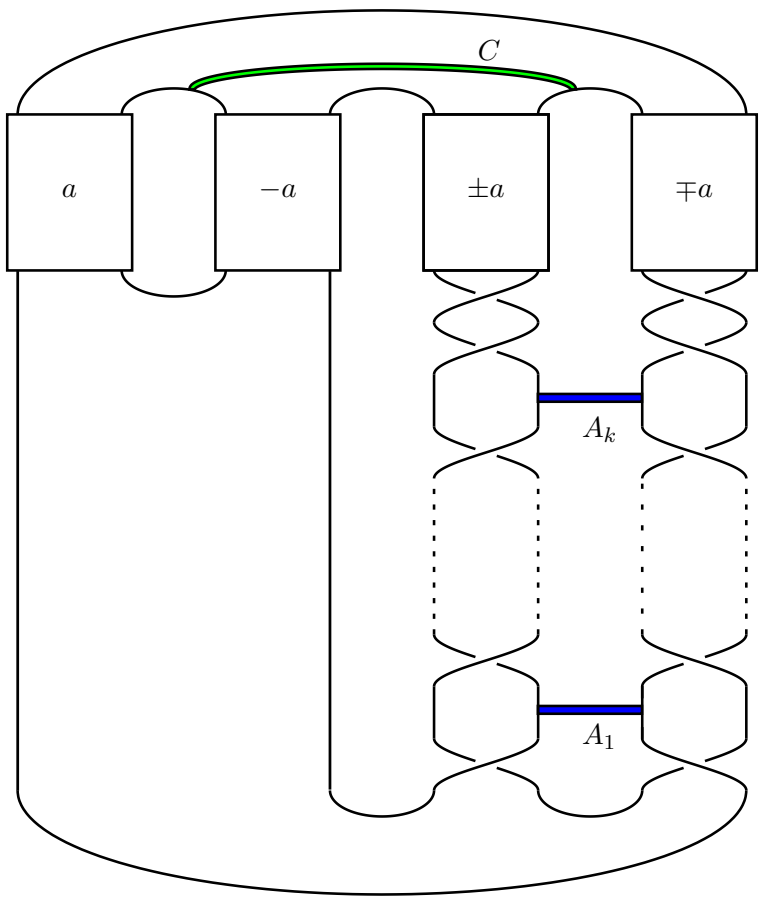
\else \vskip 5cm \fi
\begin{narrow}{0.3in}{0.3in}
\caption{
\bf{Bands $A_i$ and $C$.}}
\label{oddeven2}
\end{narrow}
\end{center}
\end{figure}
\end{proof}

\begin{Prop}
The link $P(a,-a,b,-b)$ is doubly slice when $a$ and $b$ are both odd.
\end{Prop}

\begin{proof}
We proceed in the same manner in Proposition \ref{oddeven}. The band $D$ in Figure \ref{oddodd} turns the link into the sum $T_{2,a} \# -T_{2,a} \# T_{2,b} \# -T_{2,b}$. We find band moves for this knot using Lemma \ref{torus}, and the fact that a connected sum of unknotted spheres is also unknotted. Let $a=2l+1$ and $b=2k+1$. We obtain the diagram shown in Figure \ref{oddodd2}.

\begin{figure}[htbp] 
\begin{center}
\ifpic
\small
\def\svgwidth{10cm}
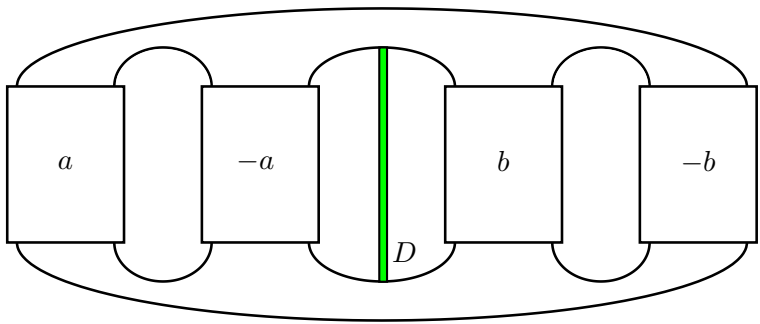
\else \vskip 5cm \fi
\begin{narrow}{0.3in}{0.3in}
\caption{
\bf{A band move on $P(a,-a,b,-b)$.}}
\label{oddodd}
\end{narrow}
\end{center}
\end{figure}

\begin{figure}[htbp] 
\begin{center}
\ifpic
\small
\def\svgwidth{10cm}
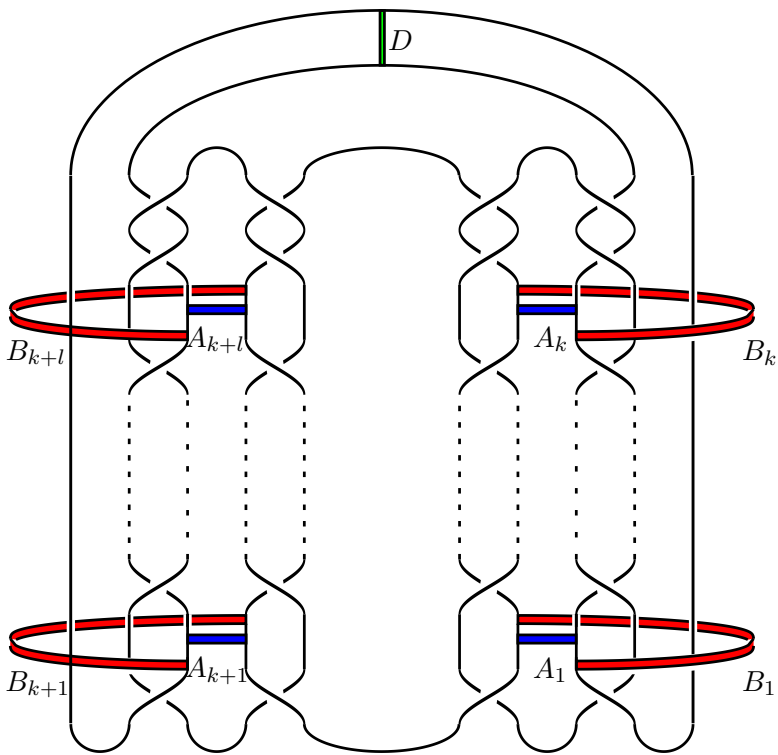
\else \vskip 5cm \fi
\begin{narrow}{0.3in}{0.3in}
\caption{
\bf{Band moves on $P(a,-a,b,-b)$.}}
\label{oddodd2}
\end{narrow}
\end{center}
\end{figure}

Setting $D=B_0$ gives the result, arguing as in Proposition \ref{oddeven} above. Figure \ref{oddodd3} just shows the bands $A_i$ and $D$. Note that after the band moves given by $A_1$ and $A_{k+1}$ all of the crossings can be removed and it is easy to see that band $D$ simply connects two components together.

\begin{figure}[htbp] 
\begin{center}
\ifpic
\small
\def\svgwidth{10cm}
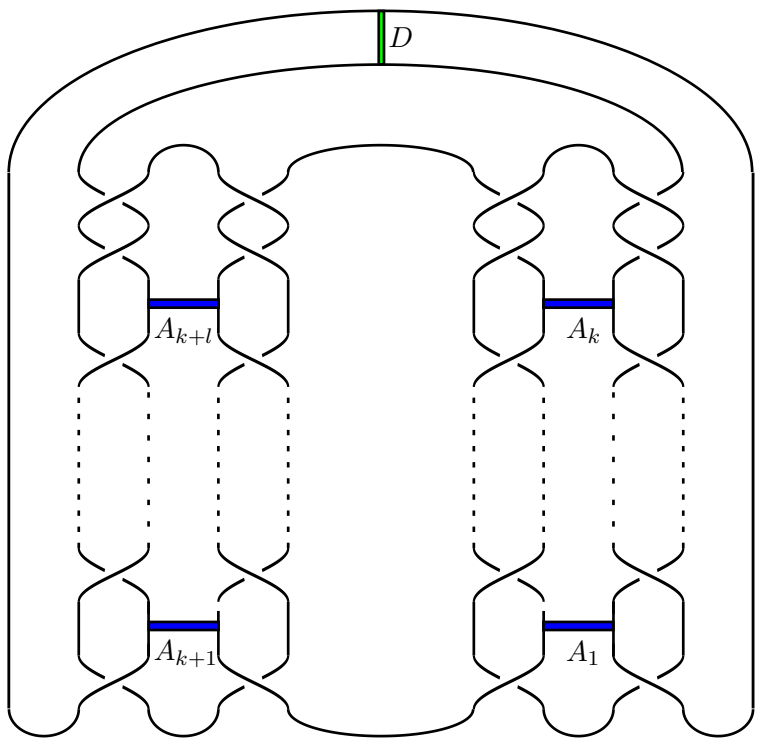
\else \vskip 5cm \fi
\begin{narrow}{0.3in}{0.3in}
\caption{
\bf{Band moves $C$ and $A_i$.}}
\label{oddodd3}
\end{narrow}
\end{center}
\end{figure}
\end{proof}


\begin{Rmk}
We have considered unoriented links but there is also a natural notion of a doubly slice oriented link as a cross-section of an oriented unknotted sphere. Every doubly slice link can be given such an orientation and, for the examples above, this orientation can be recovered from the pictures by requiring that every band respects it.

A closely related problem is to determine which (pretzel) links are smoothly doubly slice. We may use the failure of the double branched cover to embed in $S^4$ as an obstruction. However, additional complications arise -- the links $P(2,-2,3,-3)$ and $P(2,3,-2,-3)$ have the same double branched cover but only the former is doubly slice.
\end{Rmk}

Lemma \ref{dsemb} now gives all of the embeddings claimed in Theorems \ref{sumslens} and \ref{pretzelcover34}. Some of these were known already by different methods. Crisp and Hillman showed that $Y(a_1, -a_1, \ldots a_n -a_n)$ with each $a_i$ is odd and $Y(a,-a,a, \ldots , \pm a)$ always embed \cite{CH}. The manifold $Y(2,-2,3,-3)$ is known to embed by \cite{Budney}.



\subsection{Constructing embeddings using Kirby diagrams}

A Kirby diagram for $S^4$ gives a handle decomposition. Taking only some of these handles gives a 4-dimensional submanifold of $S^4$. We can find an embedding of a 3-manifold in $S^4$ from a sufficiently complicated Kirby diagram by taking the boundary of such a submanifold. Indeed, in principle, every 3-manifold which can be smoothly embedded in $S^4$ can be found in this way.

Here we will use this method to construct an embedding for another Seifert manifold.
\begin{Lem}\label{force}
Suppose $Y$ is the boundary of a Kirby diagram consisting of $4$ 2-handles. Suppose these are attached along framed knots $\gamma_i$ $(1 \leq i\leq 4)$ with the following properties:
\begin{itemize}\item The sublink given by $\gamma_1$ and $\gamma_2$ is a 0-framed unlink;
\item the sublink given by $\gamma_1$ and $\gamma_3$ is a 0-framed unlink;
\item the linking number of $\gamma_1$ and $\gamma_4$ is $\pm1$.
\end{itemize}
Then $Y$ embeds smoothly in $S^4$.
\end{Lem}

\begin{proof}
We can draw a Kirby diagram as follows. Exchange $\gamma_1$ and $\gamma_2$ for 1-handles and add 0-framed meridians to $\gamma_2$ and $\gamma_4$.

Then $\gamma_2$ and its 0-framed meridian give a canceling pair -- whenever a 2-handle crosses over $\gamma_2$ in the diagram we may change this to an undercrossing by sliding the other component over the meridian. This pair can therefore be removed.

Similarly, we can remove every crossing of $\gamma_3$ over $\gamma_4$. Since it is 0-framed and can be drawn such that it has no crossings with $\gamma_1$, we may add a canceling 3-handle.

Our diagram now consists of a 1-handle attached along $\gamma_1$, $\gamma_4$ and a 0-framed meridian of $\gamma_4$. By sliding $\gamma_4$ over this meridian, we may change any crossing of $\gamma_4$ with itself. Since the linking number of $\gamma_1$ with $\gamma_4$ is $\pm1$ we see that they give a canceling pair. After removing them, we may add a 3-handle and a 4-handle to get the standard Kirby diagram of $S^4$.

It then follows that $Y$ is the boundary of a smooth submanifold of $S^4$.
\end{proof}

We use this to describe another embedding.

\begin{Eg}
The Seifert manifold $Y(S^2;0;(4,1),(4,1),(12,-7))$ embeds smoothly in $S^4$.
\end{Eg}

We rewrite this Seifert manifold as $Y(S^2;1;(4,1),(4,1),(12,5))$. Since $12/5=[2,-3,-2]^-$ this is the boundary of the plumbing shown in Figure \ref{egpl}. We blow down the $+1$-framed curve to get the first picture in Figure \ref{early} and then perform the indicated Kirby moves.

\begin{figure}[htbp] 
\begin{center}
\ifpic
\small
\def\svgwidth{6cm}
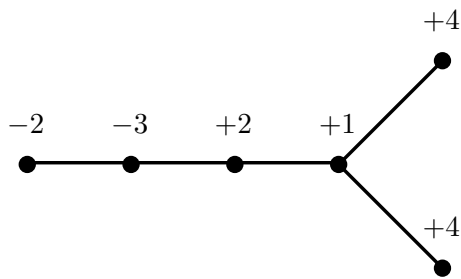
\else \vskip 5cm \fi
\begin{narrow}{0.3in}{0.3in}
\caption{
\bf{Plumbing graph for $Y(S^2;1;(4,1),(4,1),(12,5))$.}}
\label{egpl}
\end{narrow}
\end{center}
\end{figure}


\begin{figure}[htbp] 
\begin{center}
\ifpic
\def\svgwidth{6cm}
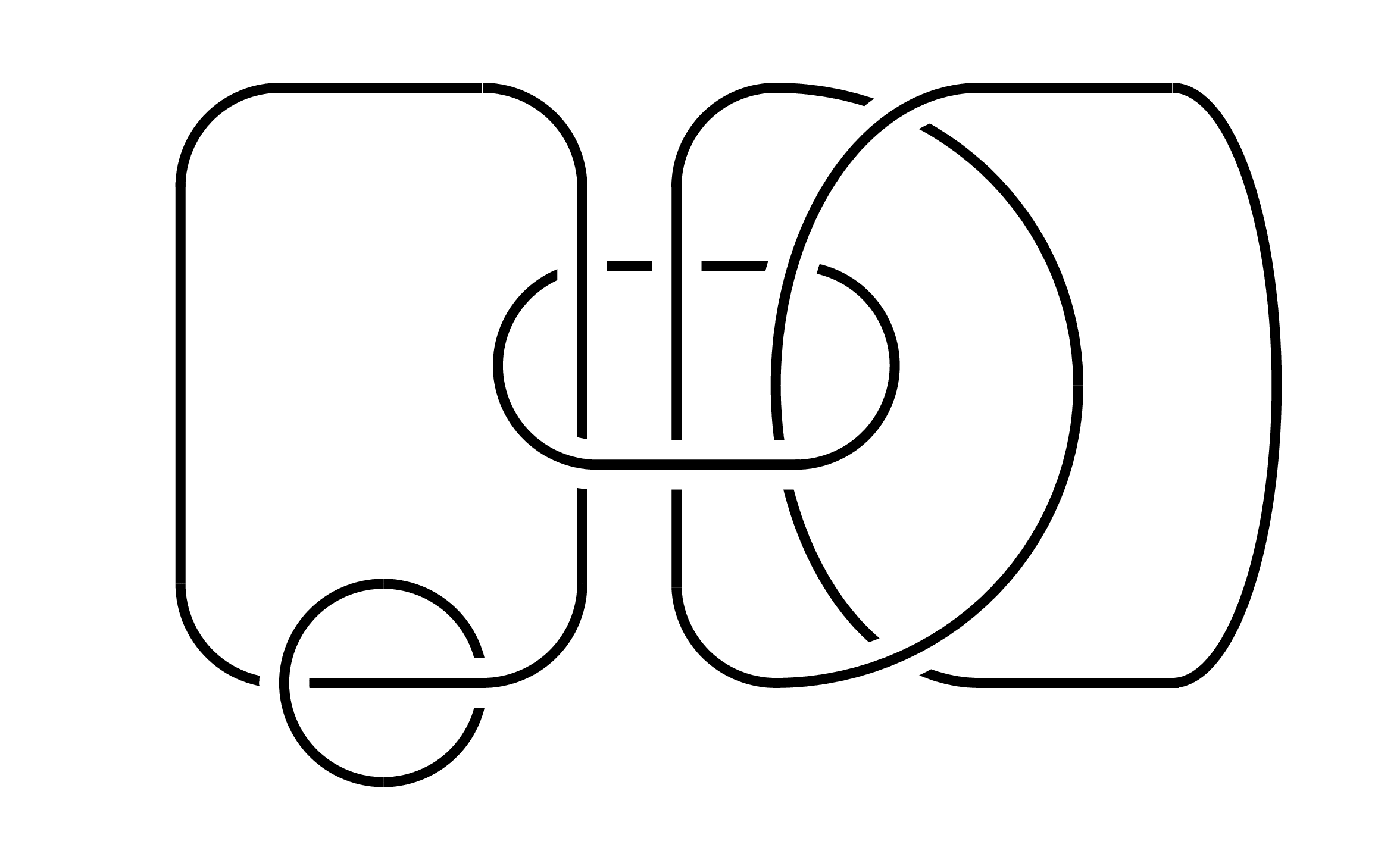
\vspace{0.4cm}

{$\downarrow$}
Handle slide.

\vspace{0.4cm}
\def\svgwidth{9cm}
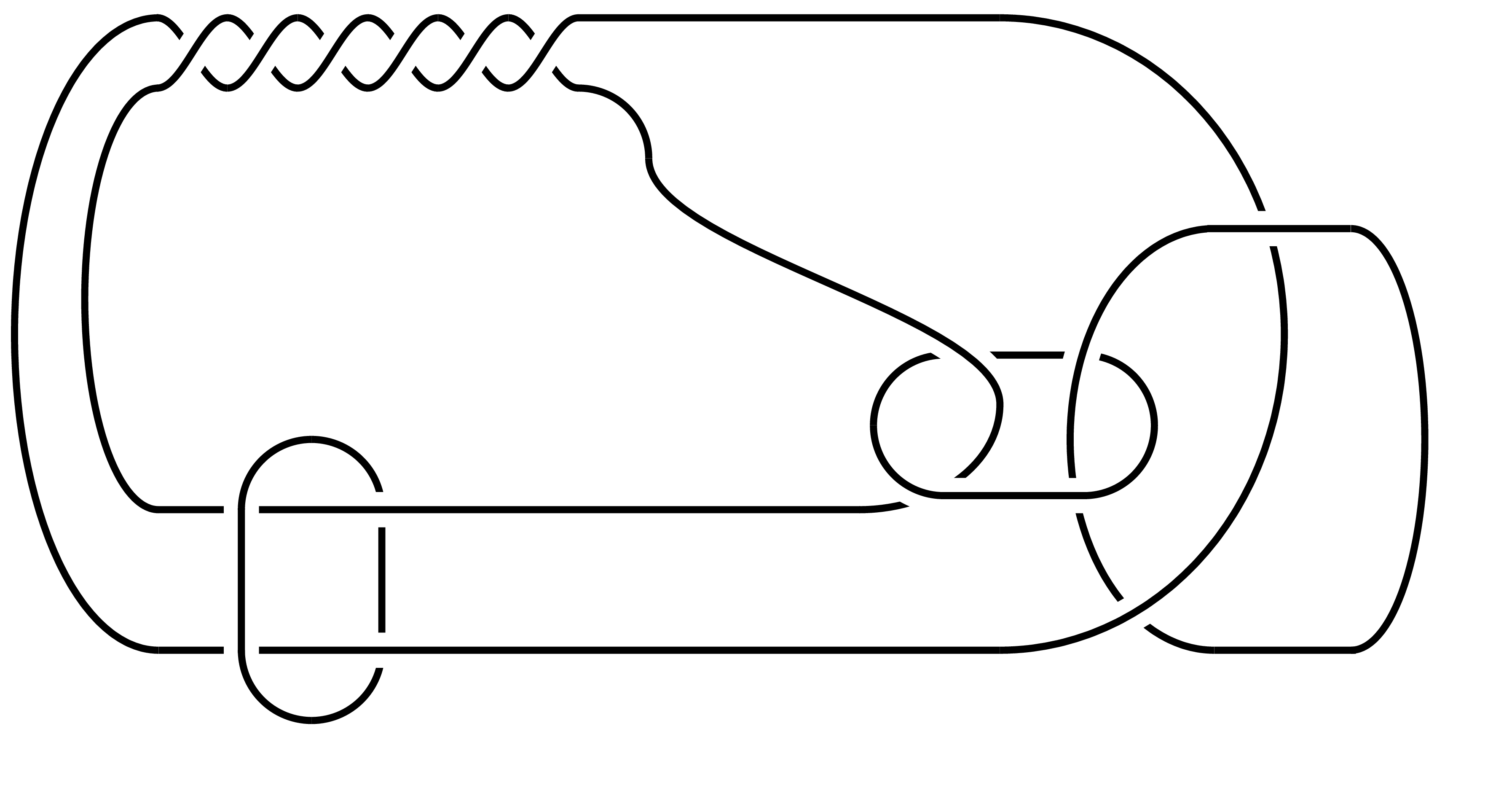
\vspace{0.5cm}

{$\downarrow$}
Blow down, handle slide.

\vspace{0.4cm}
\def\svgwidth{10cm}
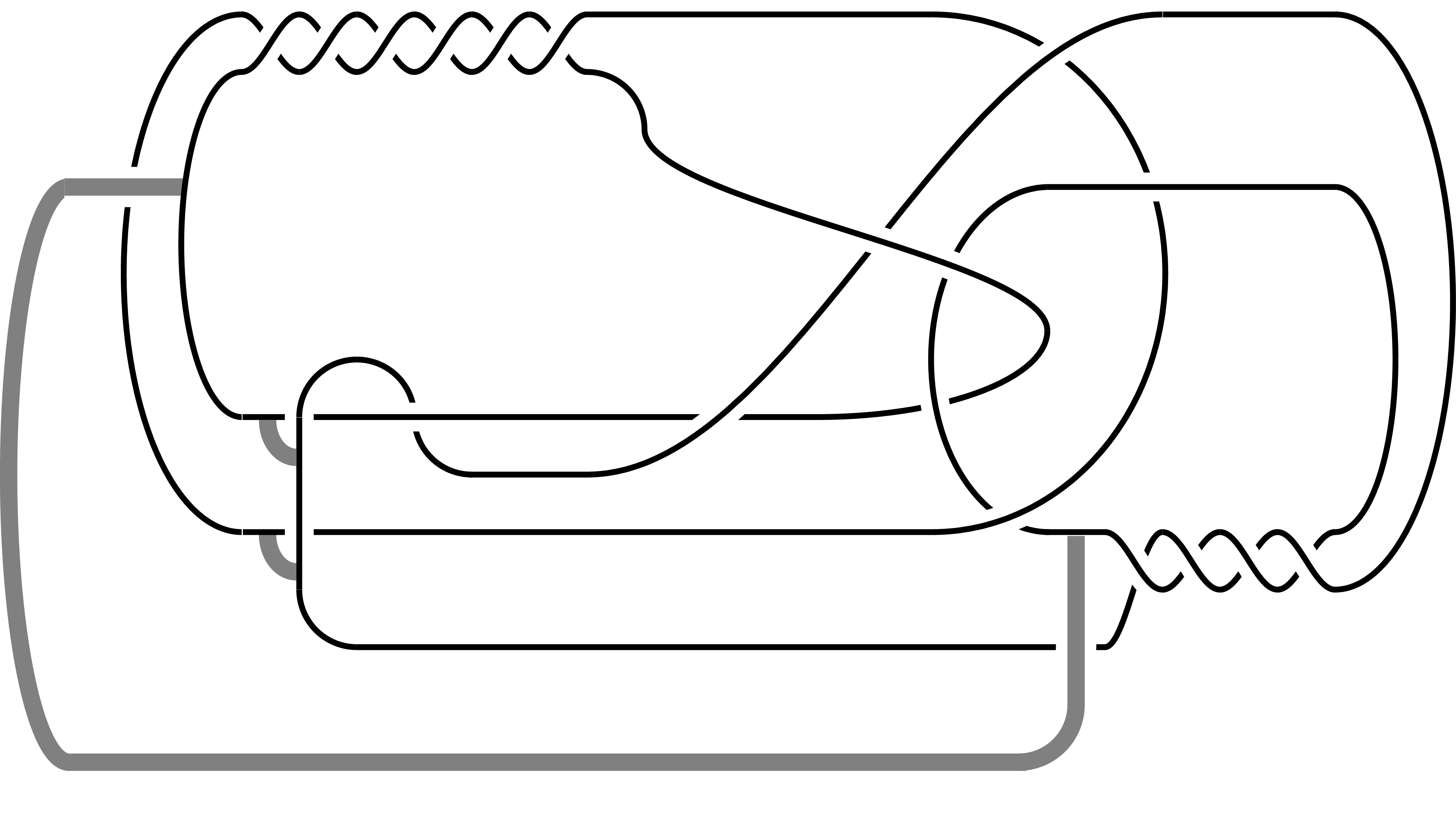

\else \vskip 5cm \fi
\begin{narrow}{0.3in}{0.3in}
\caption{
\bf{Kirby moves.}}
\label{early}
\end{narrow}
\end{center}
\end{figure}

The second diagram has a $0$-framed unknot which we think of as $\gamma_2$ to fit in with the notation of Lemma \ref{force}. The final picture is a Kirby diagram for a 4-manifold $X$ to which the lemma applies but another three handle slides are needed to draw it in the required form. The bands determining these slides are drawn. Note that there is another $0$-framed unknot which we call $\beta$ and should think of as $\gamma_1+\gamma_2$. It forms a $0$-framed unlink with $\gamma_2$.

The next handle slide uses band $A$ to slide the curve with framing $2$ over the one with framing $-4$, to get a $0$-framed curve $\gamma_3$. We then slide the $-4$ framed curve over $\beta$ using band $B$ to get $\gamma_4$ and finally use band $C$ to slide $\beta$ over $\gamma_2$. This gives a $0$-framed curve $\gamma_1$.

It is easy to check that the sublink given by $\gamma_1$ and $\gamma_2$ is a $0$-framed unlink. The linking number of $\gamma_1$ and $\gamma_4$ is a homological property of $X$ and can be computed using the intersection form of $X$. A matrix for the form can be found using the linking numbers in the final diagram in Figure \ref{early} and a simple calculation verifies that $\gamma_1$ and $\gamma_4$ have linking number $\pm 1$.

Both $\gamma_1$ and $\gamma_3$ are $0$-framed and we can see the sublink consisting of these two curves by band summing the components in the last picture of Figure \ref{early} along bands $A$ and $C$. This gives an unlink, shown in Figure \ref{unlink}, and so Lemma \ref{force} gives an embedding.

\begin{figure}[htbp] 
\begin{center}
\ifpic

\def\svgwidth{10cm}
{\scriptsize 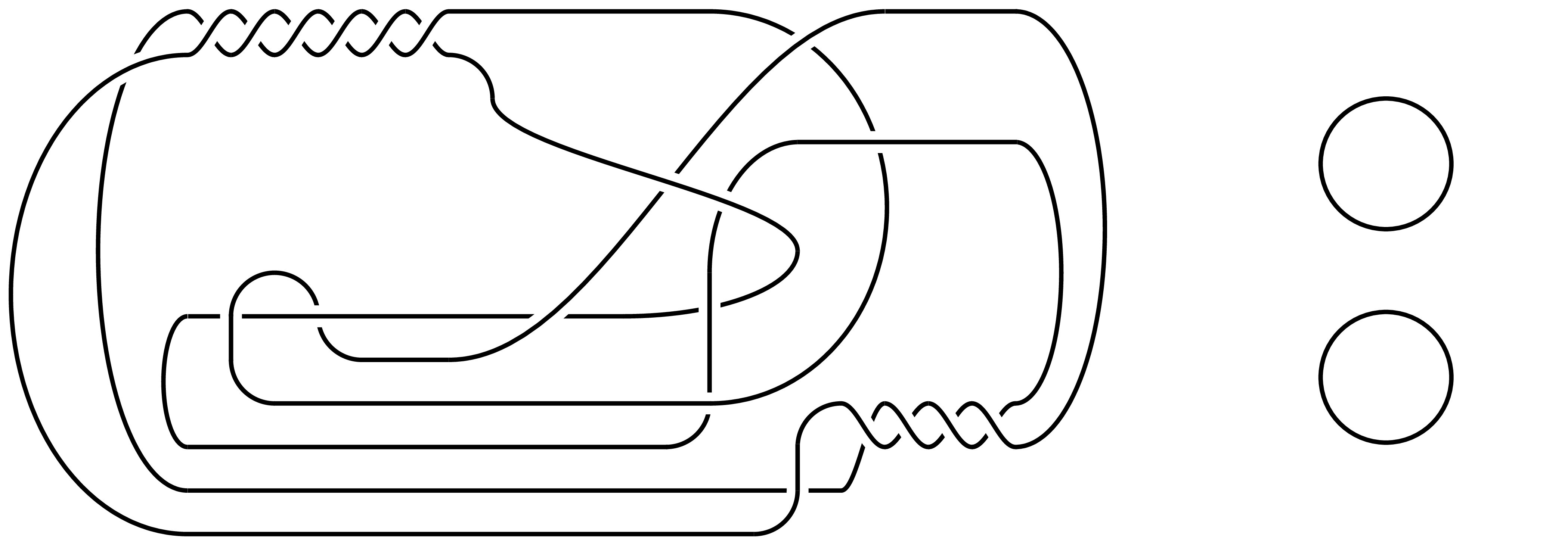}

\else \vskip 5cm \fi
\begin{narrow}{0.3in}{0.3in}
\caption{
\bf{$\gamma_1$ and $\gamma_3$ give a 2-component unlink.}}
\label{unlink}
\end{narrow}
\end{center}
\end{figure}

\newpage



\section{Diagonalisation obstruction}\label{diagonalisation}
We start by describing some purely homological properties of an embedding in $S^4$.
\begin{Lem}\label{basichom}
Suppose a 3-manifold $Y$ embeds smoothly in $S^4$. Then there is a splitting $S^4=U_1 \cup_Y -U_2$ for smooth 4-manifolds $U_i$ with boundary $Y$ such that
\begin{enumerate}
\item $H^2(Y;\mathbb{Z}) \cong H^2(U_1;\mathbb{Z})\oplus H^2(U_2;\mathbb{Z})$; \item $H^2(U_i;\mathbb{Z}) \cong H_1(U_j;\mathbb{Z})$ for $i\neq j$; \item $b_3(U_i)=0$; \item $\sigma(U_i)=0$.

\end{enumerate}

\end{Lem}

\begin{proof}
The first three statements follow by applying the Mayer-Vietoris sequence to this decomposition of $S^4$ and Alexander duality. Note in particular that the torsion subgroup $\tau H^2(Y) \cong G\oplus G$ where $G$ is the (common) torsion subgroup of $H^2(U_1)$ and $H^2(U_2)$.
Since $b_3(U_i)=0$, it follows from the exact cohomology sequence of the pair $(U_i,Y)$ that $b_1(U_i)+b_2^0(U_i)=b_1(Y)$. This implies that $b_2(U_i)=b_2^0(U_i)$ and, in particular, that the signature is zero.
\end{proof}

We will use Donaldson's theorem about 4-manifolds with definite intersection forms to obtain an obstruction.

\begin{Thm}[Donaldson \cite{donaldson}] If $W$ is a closed, oriented, smooth 4-manifold and the intersection form $Q_W: H_2(W;\mathbb{Z}) \otimes H_2(W;\mathbb{Z}) \to \mathbb{Z}$ is negative definite then $Q_W$ is diagonalisable.
\end{Thm}





\begin{Prop}\label{homologycount}Let $Y$ be a 3-manifold which bounds 4-manifolds $U, X$ where $U$ is a submanifold  of $S^4$ and $H^3(X)=0$. Let $W=X \cup_Y -U$ and let $K$ be the kernel of the inclusion map $$H_1(Y;\mathbb{Z}) \to H_1(X;\mathbb{Z})$$ and denote its rank by $k$.

If the image of $K$ in $H_1(U;\mathbb{Z})$ also has rank $k$ then $b_2(W)=b_2(X)- k$ and $\sigma(W)=\sigma(X)$.
\end{Prop}

\begin{proof}
We may calculate $b_1(W)$ using the Mayer-Vietoris sequence. The condition on $K$ implies that the three first homology terms give a short exact sequence and so $b_1(W)=b_1(X)+b_1(U)-b_1(Y)$.

Computing the Euler characteristic of $W$ gives an expression for $b_2(W)$ which may be reduced to the claimed form using the equations $b_1(U)+b_2(U)=b_1(Y)$ and $b_1(X)-b_1(Y) =-k$. These follow from Lemma \ref{basichom} and the condition on $X$.

The signatures of $W$ and $X$ are equal as $\sigma(U)=0$ by Lemma \ref{basichom}.
\end{proof}

If $X$ is chosen so that $b_2(X)-k = -\sigma(X)$ then $W$ is negative definite and Donaldson's theorem applies to show that the intersection form of $W$ is diagonal. Let $\{e_i\}$ be a basis for $H_2(W)/\text{Torsion}$ such that $e_i \cdot e_j = -\delta_{ij}$. Next we consider the induced map $\iota_* : H_2(X) \to H_2(W)$.

We may choose a basis $\{h_1, \ldots h_n\}$ of $H_2(X) \cong \mathbb{Z}^{n}$. Let $Q_X$ denote the matrix of the intersection pairing with respect to this basis.

Following \cite{lisca} we can use these to define a `subset'.
\begin{Defn}
Let $v_i = \iota_*(h_i)\in H_2(W)/\text{Torsion}$ for each $1 \leq i \leq n$. We call $S=\{v_1, \ldots v_n\}$ the subset associated to the pair $(X,U)$.

The matrix $A(S)=[e_i \cdot v_j]$ is called the matrix of $S$.
\end{Defn}
Clearly, $S$ and $A(S)$ just give different ways of recording the same information. We will switch between the two freely whenever it is convenient. Note that for the bases $\{h_i\}$ and $\{e_j\}$ the map $\iota_*$ is represented by the matrix $A(S)^t$.

An important feature of the subset $S$ is that it encodes information about the manifold $X$ and the image of the torsion subgroup $\tau H^2(U)$ of $H^2(U)$ in $H^2(Y)$.



\begin{Lem}\label{factor}Let $W=X \cup_Y -U$ where $U$ is a smooth 4-dimensional submanifold of $S^4$ and suppose $W$ is negative definite. Choose a basis for $H_2(X)$ and let $S$ be the associated subset.

The matrix $A(S)$ is such that $Q_X=-A(S)A(S)^t$.
\end{Lem}


\begin{proof}
The homology classes in $H_2(X)$ are represented by embedded surfaces and the intersection form counts the signed intersection points of these surfaces. If surfaces $\{\alpha_i\}$ represent classes in $H_2(X)$ then the same surfaces sit inside $W$ to represent the images of these homology classes under the inclusion induced map. Since the intersection points are the same, $Q_X(h_i , h_j) = -Id(v_i , v_j)$. The matrix $A(S)^t$ represents the inclusion map so $v_i = A(S)^t h_i$ and so for every pair $h_i,h_j$, $Q_X(h_i,h_j) = -A(S)A(S)^t(h_i,h_j)$.
\end{proof}

\begin{Thm}\label{mainthm}Let $U$ be a submanifold of $S^4$ and $X$ be such that $H^3(X;\mathbb{Z})=0$, $H_1(X;\mathbb{Z})$ is torsion-free and the matrix $Q_X$ is non-singular. Suppose $\partial X = -\partial {U} =Y$ and that $W=X\cup_Y U$ is negative-definite. Let $S$ be the associated subset.

There is an isomorphism between the torsion subgroup of the image of the restriction map $H^2(U) \to H^2(Y)$ and $\left(\frac{\operatorname{im} A(S)}{\operatorname{im}Q_X}\right)$ and this is facilitated by the inclusion induced map $\delta: H^2(X) \to H^2(Y)$.
\end{Thm}

\begin{proof}

We follow the approach of \cite[Proposition 2.5]{GJ}. Consider the following diagram with the maps of cohomology induced by the inclusion $(X,Y) \hookrightarrow (W,U)$:
$$\begin{CD}\label{commd}  @>>> H^2(W,U) @>\alpha>> H^2(W) @>\beta>> H^2(U) @>>> H^3(W,U) \\
@. @V\iota_1 \cong VV @V \iota_2 VV @V\iota_3VV @V\iota_4 \cong VV \\
 @>>> H^2(X,Y) @>\gamma>> H^2(X) @>\delta>> H^2(Y) @>>> H^3(X,Y) .\end{CD}$$

The rows of this diagram are exact and it is commutative since all of the maps are given by restriction.


Given the basis $\{h_i\}$ for $H_2(X)$ we may choose the dual and Poincar\'e dual bases for $H^2(X)$ and $H^2(X,Y)$. With these choices the map $\gamma$ is represented by $Q_X$. This sets up an identification of a subgroup of $H^2(Y)$ with $\operatorname{coker} Q_X$ via $\delta$. Since $\det Q_X \neq 0$, this lies in the torsion subgroup of $H^2(Y)$ and the fact that $H_1(X)$ is torsion-free shows that this gives the whole torsion subgroup.

We are interested in the image of $\iota_3$. This has a subgroup given by the image of $\iota_3 \circ \beta$. Since this is the same as the image of $\delta \circ \iota_2$ it is a finite group.

We may choose the dual basis to $\{e_i\}$ for $H^2(W)/\text{Torsion}$ so that the map $\iota_2$ is represented by the matrix $A(S)$. Note that since $H^2(X)$ is free abelian any torsion in $H^2(W)$ must map trivially.

The image of $\delta \circ \iota_2$ is therefore isomorphic to $\left(\frac{\operatorname{im} A(S)}{\operatorname{im}Q_X}\right)$.

To see that this gives the entire torsion subgroup of the image of $\iota_3$, we compare the orders. By Lemma \ref{factor}, $Q_X=-A(S)A(S)^t$ and so the order of this subgroup is $|\det A(S)|$. By Lemma \ref{basichom} the torsion image of $\iota_3$ also has order given by the square root of $|\operatorname{coker} Q_X|$.
\end{proof}


\begin{Rmk}
The assumption in Theorem \ref{mainthm} that $U$ is a submanifold of $S^4$ can sometimes be weakened. When $Y$ is a rational homology sphere and $U$ any rational ball this result is Theorem 3.5 of \cite{GJ}.
\end{Rmk}


\subsection{Definite 4-manifolds bounded by Seifert manifolds}
We will apply this result to obtain obstructions to embedding Seifert manifolds. To do this, we describe the relevant negative definite 4-manifolds.

Recall that a negative definite plumbing bounded by the lens space $L(p,q)$ can be constructed by plumbing on a linear graph with weights given by the negative continued fraction.
A similar construction works when $Y$ is a Seifert manifold with base $S^2$ and $e(Y) >0$. It may be arranged that the Seifert invariants are of the form $(a_i,b_i)$ with $a_i>-b_i>0$. A weighted graph, which yields a plumbing with boundary $Y$, can be obtained by taking a central vertex weighted by the central framing and attaching legs with weights according to the negative continued fractions of $a_i/b_i$. It is shown in \cite{NR} that this is negative definite whenever $e(Y) >0$ and negative semi-definite when $e(Y)=0$.

To get a surgery picture for a Seifert manifold with a different base surface, we modify the diagram at the central curve. Figure \ref{baset2} shows how to add an orientable handle and Figure \ref{baserp2} how to add an $\mathbb{RP}^2$ summand to the base. There are other equivalent pictures -- some are shown in \cite[Appendix]{CH}.

\begin{figure}[htbp] 
\begin{center}
\ifpic
\def\svgwidth{4cm}
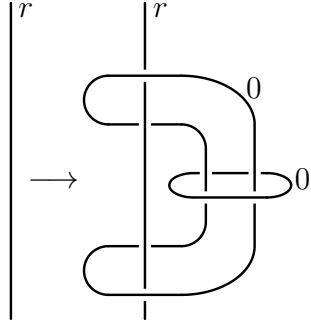
\else \vskip 5cm \fi
\begin{narrow}{0.3in}{0.3in}
\caption{
\bf{Adding an $S^1 \times S^1$ summand.}}
\label{baset2}
\end{narrow}
\end{center}
\end{figure}

\begin{figure}[htbp] 
\begin{center}
\ifpic
\def\svgwidth{4cm}
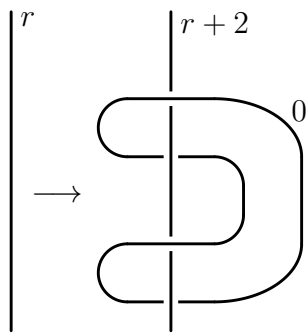
\else \vskip 5cm \fi
\begin{narrow}{0.3in}{0.3in}
\caption{
\bf{Adding an $\mathbb{RP}^2$ summand.}}
\label{baserp2}
\end{narrow}
\end{center}
\end{figure}

The construction of negative-definite manifolds with Seifert boundaries can be extended. The intersection forms depend primarily on the Seifert invariants, not the base surface.

\begin{Prop}\label{definites}
Let $Y_F=Y(F;r;(a_1,b_1), \ldots (a_n,b_n))$ where $F$ is a closed surface and $L=-\#_{i=1}^nL(a_i,b_i)$.

There are 4-manifolds $X_L$ and $X_F$ with boundaries $L$ and $Y_F$ respectively. The 4-manifolds $X_L$ and $X_{S^2}$ are obtained by plumbing and the intersection form $Q_{X_F}$ is equivalent to $Q_{X_{S^2}}$ if $F$ is orientable and to $Q_{X_L}$ otherwise.

In addition, $X_L$ can always be chosen to be negative definite and $X_{S^2}$ can be chosen to be negative definite if $e(Y_{S^2}) >0$ and semi-definite if $e(Y_{S^2})=0$.
\end{Prop}

\begin{proof}
The manifolds $X_L$ and $X_{S^2}$ are described above.

We get a Kirby diagram for $X_F$ by modifying the diagram for $X_{S^2}$. We add 1-handles in place of the new $0$-framed 2-handles in Figures \ref{baset2} and \ref{baserp2}.

The intersection forms of these manifolds are easy to describe. There are two cases, depending on whether the base surface is orientable or not, but the intersection form does not depend on the genus. When the base surface $F$ is orientable, there is a canonical basis for $H_2(X;\mathbb{Z})$ given by the cores of the 2-handles. The intersection form is given by the incidence matrix of the plumbing graph, obtained by ignoring any 1-handles. This gives a manifold $X_F$ with the same intersection form as $X_{S^2}$.

When the base surface is non-orientable, the central curve does not contribute to the second homology. The intersection form is given by the other 2-handles. This is the same as the intersection form of $X_L$.
\end{proof}



We can now apply Theorem \ref{mainthm}.

\begin{Cor}\label{diag0}Let $Y$ be a connected sum of lens spaces or a Seifert manifold with orientable base orbifold and $e(Y) >0$, which embeds smoothly in $S^4$ and let $X$ be the negative definite 4-manifold with boundary $Y$ from Proposition \ref{definites}. Then there are $b_2(X) \times b_2(X)$ integer matrices $A_1,A_2$ such that $A_i A_i^t = -Q_X$ for $i=1,2$. Viewing $A_1,A_2$ and $Q_X$ as maps $\mathbb{Z}^{b_2(X)} \to \mathbb{Z}^{b_2(X)}$ let $H_i = \frac{\operatorname{im}A_i}{\operatorname{im}Q_X}$ be subgroups of $\operatorname{coker} Q_X $.

Then $\operatorname{coker} Q_X = H_1 \oplus H_2$ and $H_1 \cong H_2$.
\end{Cor}
\begin{proof}The embedding produces a splitting $S^4 = U_1 \cup_Y -U_2$. Applying Theorem \ref{mainthm} to $W_i=X \cup_Y -U_i$ gives the matrices $A_i$ and identifies the images of the restriction maps $\tau H^2(U_i) \to H^2(Y)$ with $\frac{\operatorname{im}A_i}{\operatorname{im}Q_X}$ using the map $\delta$.

The result can now easily be deduced from Lemma \ref{basichom} since the isomorphism $H^2(U_1) \oplus H^2(U_2) \to H^2(Y)$ is induced by the inclusions of $Y$ into each $U_i$.
\end{proof}
\begin{Rmk}\label{zhsf}
When $Y$ is an integral homology sphere then $H_1 = H_2 =\{0\}$ and it is possible to just take $A_1=A_2$.

Otherwise, this condition implies that $A_1$ and $A_2$ must be different. In particular, since $H_i$ is the subgroup of $\operatorname{coker}Q_X$ generated by the columns of $A_i$, the column spaces of the matrices must be different.
\end{Rmk}

\begin{Rmk}
This result holds for any negative definite 4-manifold $X'$ provided the inclusion map $H_1(Y;\mathbb{Q}) \to H_1(X',\mathbb{Q})$ is an isomorphism.
\end{Rmk}

\begin{Cor}\label{diag1}Let $Y$ be a Seifert manifold with orientable base orbifold and $e(Y)=0$. If $X$ is the semi-definite 4-manifold with boundary $Y$ from Proposition \ref{definites} then there is a $b_2(X) \times b_2(X)-1$ integer matrix $A$ such that $A A^t =-Q_X$.
\end{Cor}
\begin{proof}
The embedding splits $S^4$ as $U_1 \cup_Y -U_2$. The kernel $K$ of the map $H_1(Y;\mathbb{Z}) \to H_1(X;\mathbb{Z})$ has rank one. By Lemma \ref{basichom} the inclusion maps give an isomorphism $H_1(Y;\mathbb{Z}) \cong H_1(U_1;\mathbb{Z}) \oplus H_1(U_2;\mathbb{Z})$ and hence $K$ must map non-trivially into $H_1(U_i;\mathbb{Z})$ for some $i$. At least one of $X \cup_Y -U_1$ and $X \cup_Y -U_2$ is negative definite so the result then follows by applying Lemma \ref{factor}.
\end{proof}

While a Seifert manifold with a non-orientable base surface is also the boundary of a canonical negative definite manifold $X$, the first homology of $X$ is not torsion-free. However, we may modify the proof of Theorem \ref{mainthm} to recover a result slightly weaker than Corollary \ref{diag0}.

\begin{Cor}\label{diagnon}Let $Y$ be a Seifert manifold with non-orientable base orbifold $P_k$, which embeds smoothly in $S^4$ and let $X$ be the negative definite 4-manifold from Proposition \ref{definites}. Then there are $b_2(X) \times b_2(X)$ integer matrices $A_1,A_2$ such that $A_i A_i^t = -Q_X$ for $i=1,2$. Viewing $A_1,A_2$ and $Q_X$ as maps $\mathbb{Z}^{b_2(X)} \to \mathbb{Z}^{b_2(X)}$ let $H_i = \frac{\operatorname{im}A_i}{\operatorname{im}Q_X}$ be subgroups of $\operatorname{coker} Q_X $.

Then $\operatorname{coker} Q_X \cong H_i \oplus H_i$ for $i=1,2$ and $|H_1 \cap H_2| \leq 2$.
\end{Cor}

\begin{proof}

As before $S^4= U_1 \cup_Y -U_2$ and this gives subsets with associated matrices $A_1$ and $A_2$. 
Let $t$ be the unique element of order of two in $H^2(X) \cong \mathbb{Z}^{b_2(X)} \oplus \mathbb{Z}/2$. The commutative diagram from the proof of Theorem \ref{mainthm} can be extended, with $i=1,2$.

$$\begin{CD}\label{commd2}  @>>> H^2(W_i,U_i) @>\alpha^i>> H^2(W_i) @>\beta^i>> H^2(U_i) @>>> H^3(W_i,U_i) \\
@. @V\iota^i_1 \cong VV @V \iota^i_2 VV @V\iota^i_3VV @V\iota^i_4 \cong VV \\
 @>>> H^2(X,Y) @>\gamma>> H^2(X) @>\delta>> H^2(Y) @>>> H^3(X,Y) \\
 @.@.  @Vq_1VV @ Vq_2VV @. \\
 @. @. \dfrac{H^2(X)}{\text{Torsion}} @>\delta'>> \dfrac{H^2(Y)}{\langle\delta(t)\rangle} @. \end{CD}$$

With respect to the appropriate bases, $q_1 \circ \gamma$ is represented by $Q_X$. Note that the image of this composition is the kernel of $\delta'$ so there is an isomorphism between $\operatorname{coker} Q_X$ and $\operatorname{im} \delta'$. The torsion of $H^2(W_i)$ maps trivially under $q_1 \circ \iota^i_2$ so we can identify this map with the matrix $A_i$. The group $H_i$ can now be seen as the image of $\delta' \circ q_1 \circ \iota^i_2$.

By Proposition \ref{definites}, $$\operatorname{coker} Q_X= \bigoplus_{i=1}^n \mathbb{Z}/{a_n},$$ where the $a_i$ come from the Seifert invariants of $Y$. We may order the $a_i$ by writing each as $a_i=2^{t_i} s_i$ with $s_i$ odd and arranging that $t_1 \geq t_2 \geq \ldots \geq t_n$. With this ordering, \cite[Lemma 3.4]{CH} tells us that $$\tau H^2(Y) = \left(\bigoplus_{i=3}^n \mathbb{Z}/{a_i} \right)\oplus \mathbb{Z}/{2a_1} \oplus \mathbb{Z}/{2a_2} \text{ or }\left(\bigoplus_{i=2}^n \mathbb{Z}/{a_i}\right) \oplus \mathbb{Z}/{4a_1}.$$

By Lemma \ref{basichom}, this torsion subgroup is of the form $H \oplus H$ so we may assume the former holds. Decomposing $\tau H^2(Y)$ as a direct sum of cyclic groups of prime power order we see that it is $$\mathbb{Z}/{2^{t_1+1}} \oplus \mathbb{Z}/{2^{t_1+1}} \oplus K \oplus K,$$ for some $K$ while $\operatorname{coker} Q_X$ is $$\mathbb{Z}/{2^{t_1}} \oplus \mathbb{Z}/{2^{t_1}} \oplus K \oplus K.$$

Also $H_i = \operatorname{im} q_2 \circ \iota_3^i \circ \beta^i$ is a subgroup of $q_2(\mathbb{Z}/{2^{t_1+1}} \oplus K)$.
Since this is a square root order subgroup of $\operatorname{coker} Q_X$, it follows that this cokernel is isomorphic to $H_i \oplus H_i$.

To see $H_1$ and $H_2$ have the required intersection, note that they are images of maps which factor through $\iota^1_3$ and $\iota^2_3$. The images of these maps have trivial intersection by Lemma \ref{basichom}. Since $q_2$ takes the quotient by a subgroup of order two, $H_1$ and $H_2$ have at most two points of intersection in $\dfrac{H^2(Y)}{\langle \delta(t)\rangle}$.
\end{proof}



\section{Linear subsets}\label{lins}

In this section we describe the combinatorics necessary to prove Theorems \ref{sumslens}, \ref{non-ori} and \ref{e=0}. Let $\mathbb{D}^n$ be the lattice $\mathbb{Z}^n=\langle e_1, \ldots e_n \rangle$ with respect to the product given by $-Id$.
\begin{Defn}A subset $S=\{ v_i\}$ of $\mathbb{D}^n$ is called linear if \begin{equation}\label{sumlenscondition}v_i \cdot v_j = \begin{cases} -a_i \leq -2 &\text{ if } i=j
\\ 0 \text{ or } 1 &\text{ if } |i-j|=1
\\ 0 &\text{ if } |i-j|>1.
\end{cases}\end{equation}
\end{Defn}

A weighted graph $\Gamma(S)$ can be associated to every linear subset $S$ as follows. For each element $v_i$ there is a vertex with weight $v_i \cdot v_i$ and there is an edge connecting the vertices corresponding to $v_i$ and $v_j$ if and only if $v_i \cdot v_j=1$. We will use the same notation for both the vector $v_i$ and the corresponding vertex. Define $c(S)$ to be the number of connected components of the graph $\Gamma(S)$. 

Let $Q_{\Gamma}=-A(S)A(S)^t$ be the incidence matrix of $\Gamma$. 

Define $$G(S) = \frac{\mathbb{Z}^n}{\operatorname{im}Q_{\Gamma}} \text{ and } H(S) = \frac{\mathbb{Z}^n}{\operatorname{span} S} \cong \frac{\operatorname{im}A(S)}{\operatorname{im}Q_{\Gamma}}.$$

\begin{Defn}A linear subset $S$ is called a linear double subset if $G(S) \cong H(S) \oplus H(S)$.
\end{Defn}

Linear subsets were studied extensively by Lisca \cite{lisca}, \cite{lisca2}. It will be useful to review some of these ideas.

A pair of vectors $v, v'$ are called linked if there is some unit basis vector $e_i$ in $(\mathbb{Z}^n, -Id)$ such that $v \cdot e_i$ and $v' \cdot e_i$ are both nonzero.
A subset $S$ is called irreducible if for any pair of vectors $v,v' \in S$ there is a sequence $v=w_1, \ldots , w_k=v'$ such that each $w_i$ is linked to $w_{i+1}$. In \cite{lisca2} irreducible linear subsets were called good.
\begin{Lem}
Let $S$ be a linear subset. If $S$ is not irreducible then $S= \cup_i T_i$ where each $T_i$ is irreducible and consists of $n_i$ vectors which are supported on $n_i$ of the basis vectors $\{e_j\}$.
\end{Lem}
\begin{proof}
This is proved on page 2162 of \cite{lisca2}.
\end{proof}

We now look at how to describe the groups $H(S)$ and $G(S)$ in terms of the decomposition into irreducible subsets.

For a linear subset $S$, the connected components of the graph $\Gamma(S)$ are all linear weighted trees.
If $Q_i$ denotes the incidence matrix of the $i^{th}$ tree we can arrange that $Q_{\Gamma}$ is the diagonal block matrix $\operatorname{diag}(Q_1, \ldots Q_h)$. If the subset $S$ gives matrix $A$ then this has the form $\operatorname{diag}(A_1, \ldots A_k)$ where each $A_j$ comes from an irreducible subset $T_j$.
The group $H(S)$ also splits up as a direct sum with summands of the form $$H(T_i) = \frac{\operatorname{im}A_i}{\operatorname{im}Q_{i_1} \oplus \ldots \oplus Q_{i_{c(T_i)}}}.$$

\begin{Prop}\label{irred}
Let $S$ be a linear double subset and suppose $S$ decomposes as $S= \cup_{i=1}^k T_i$ where each $T_i$ is irreducible. Then $H(T_i)$ is a square root order direct summand of $G(T_i)$ for each $i$.

In addition, if $c(T_i)=2$ then $T_i$ is also a linear double subset.
\end{Prop}

\begin{proof}

For each $1 \leq j \leq k$, let $G'=\bigoplus_{i \neq j} G(T_i)$ and $H' =\bigoplus_{i\neq j} H(T_i)$.
Consider the following diagram
$$\begin{CD}
@. 0 @. 0 @. \\
@. @VVV @VVV @.@.\\
0 @>>> H(T_j) @ >\iota>> G(T_j) @>>> \dfrac{G(T_j)}{H(T_j)} @>>> 0 \\
@. @V\uparrow VV @V\uparrow VV @.@.\\
0 @>>> H(T_j) \oplus H' @ >\leftarrow>> G(T_j)\oplus G' @>\leftarrow>> H(T_j) \oplus H' @>>> 0 \\
@. @V\uparrow VV @V\uparrow VV @.@.\\
0 @>>> H' @ >>> G' @>>> \dfrac{G'}{H'} @>>> 0 \\
@. @VVV @VVV @.@. \\
@. 0 @. 0 @.\end{CD}.$$

It follows from the description above that this diagram commutes.
The rows and columns are exact, with the obvious inclusion and quotient maps, and the first two columns and second row split.

There is then a map $\rho$ from $G(T_j) \to H(T_j)$. It is not hard to check that this splits the first row as well.

Thus $G(T_j)\cong H(T_j) \oplus K_j$ for some $K_j$ with the same order as $H(T_j)$. If $c(T_j)=2$ then $G_j$ can be written as a sum of two cyclic groups. Both $H(T_j)$ and $K_j$ must be cyclic groups.
\end{proof}

The following special case can be observed immediately.
\begin{Cor}\label{no1s}If $S$ is a linear double subset then every irreducible $T_i$ has $c(T_i) \geq 2$.
\end{Cor}

We review a few more important notions from \cite{lisca}, namely the quantity $I(S)$, contractions of subsets and bad components.
\begin{Defn}
Let $S=\{v_i\}_{i=1}^m$ be a subset of $\mathbb{D}^m$. Define $$I(S) = \sum_{i=1}^m{-v_i \cdot v_i -3}.$$
\end{Defn}
Note that $I(S)$ can be computed from weights of the graph $\Gamma(S)$. We will also use $I(C)$ when $C$ is a connected component of the graph by just summing over vectors corresponding to vertices in this component.

\begin{Defn}Let $S \subset \mathbb{D}^m$ be a subset $\{v_i\}$ for which $|v_i \cdot e_j| \leq 1$ for each $i,j$. If there are $j,s,t$ such that $|v_i \cdot e_j| =1 $ if and only if $i\in\{s,t\}$ and $v_t \cdot v_t < -2$ then the subset $S' = S \backslash \{v_s, v_t\} \cup v'_t$ of $\mathbb{D}^{n-1}$ considered as the span of $\{e_k\}_{k\neq j}$ and where $v'_t$ is obtained from $v_t$ by removing the $e_j$ component is said to be obtained via a contraction of $S$.

Conversely, $S$ is called an expansion of $S'$.
\end{Defn}

In the particular case where $v_s$ is a leaf of the graph and $v_s \cdot v_s=-2$ we will say that $S$ is an expansion of $S'$ by a final $(-2)$ vector.

\begin{Defn}
Let $S'$ be a linear subset of $\mathbb{D}^m$. Suppose that the subset $\{v_{s-1}, v_s , v_{s+1}\}$ is a connected component $C'$ of $\Gamma(S')$ and that there are $i,j$ such that $v_{s-1}$ and $v_{s+1}$ are both of the form $\pm(e_i \pm e_j)$ and $v_s \cdot v_s < -2$.

Let $S$ be any subset which is obtained from $S'$ by a sequence of expansions by final $(-2)$ vectors which belong to the connected component $C$ of $\Gamma(S)$ corresponding to $C'$.

The component $C$ is called a bad component of $S$.

We then define $b(S)$ to be the number of bad components of $S$.
\end{Defn}


Note that the conditions on $S'$ mean that, up to reordering or a change of sign, $v_{s-1}= e_i -e_j$, $v_s = e_j + \ldots$ and $v_{s+1}=-e_i-e_j$. Since every other element of $S'$ has product zero with $v_{s-1}$ and $v_{s+1}$, none contains a nonzero multiple of $e_i$ or $e_j$.

We may form a new subset $S''$ of $\mathbb{D}^{m-2}$ from $S'$ by deleting the elements $v_{s-1}$ and $v_{s+1}$ from the subset and deleting the basis vectors $e_i$ and $e_j$.

Note that the bad component $C'$ of $S'$ is necessarily given by a chain of length three with weights $-2,-n-1,-2$ for some $n\geq2$. The corresponding component $C''$ of $S''$ is simply an isolated vertex with weight $-n$.
We will call $C''$ in $S''$ the reduced component corresponding to $C$ in $S$.

We summarise the relevant features of bad components below.
\begin{Prop}\label{factsaboutBC}
Let $S$ be a linear subset with a bad component $C$. Suppose $C=\{v_1, \ldots v_s\}$ and $S \backslash C = \{w_1, \ldots w_r\}$. Then, possibly after reordering $\{e_i\}$, $w_i \cdot e_j=0$ for all $j < s$. Also, there is some $1<t<s$ such that whenever $j \geq s$ and $v_i\cdot e_j \neq 0$ then $i=t$.

Furthermore, the plumbing 4-manifold defined by the component $C$ has boundary $L(m^2n, mnk+1)$, where $-n$ is the weight on the reduced component corresponding to $C$ and $m,k$ are coprime integers with $m>k>0$. In addition, $I(C)=n-4$.
\end{Prop}
\begin{proof}
In the case where $s=3$, this description follows from the discussion above. For $s>3$, it follows from the definition of expansion by a final $(-2)$ vector. The claim that $C$ represents $L(m^2 n, mnk+1)$ is the content of \cite[Lemma 3.3]{lisca2}.
It is apparent that $I(C)$ does not change under expansion by a final $(-2)$-vector. A simple calculation verifies the dependence on $n$.
\end{proof}


The key results concerning bad components are the following:

\begin{Prop}\label{twoBC}
Let $S$ be a linear double subset with $c(S)=2$. Then $S$ does not have a bad component.
\end{Prop}
\begin{proof}Suppose $S$ has a bad component so $G(S) \cong \mathbb{Z}/{m^2 n} \oplus \mathbb{Z}/{k}$ for some $k$. This has square order so there is some $q$ so that $nk=q^2$. Every element of $H(S) \oplus H(S)$ has order dividing $mq$ and this implies that $k=mq=m^2n$.

Then we may assume that $G(S) = \left(\mathbb{Z}/{m^2 n}\right)^2$ for some $n,m \geq 2$.

We will show that if $S$ has a bad component then every element of $H(S)$ has order dividing $mn$ and so $H(S) \oplus H(S)$ is not $\left(\mathbb{Z}/{m^2 n}\right)^2$.

Letting $r,s$ be as in Proposition \ref{factsaboutBC}, $H(S)$ is the subgroup of $\frac{\mathbb{Z}^{s+r}}{\operatorname{im} Q_{\Gamma}}$ generated by the columns of $A(S)$. Our aim is to show that multiplying each column by $mn$ gives an element of $\operatorname{im} Q_{\Gamma}$.

The matrix $Q_{\Gamma}$ can be split up as $Q_1 \oplus Q_2$ where $Q_1$ is the $s \times s$ matrix corresponding to the bad component $C$ and $Q_2$ is the $r \times r$ matrix coming from the other component.
A column $$v=\begin{pmatrix}v_1 \\ \vdots \\ v_s \\ v_{s+1} \\ \vdots \\ v_{s+r} \end {pmatrix} \text{ is trivial if } \begin{pmatrix}v_1 \\ \vdots \\ v_s  \end {pmatrix} = Q_1 y_1 \text{ and } \begin{pmatrix} v_{s+1} \\ \vdots \\ v_{s+r} \end {pmatrix} = Q_2 y_2$$ for some $y_1, y_2$.

These two conditions can be checked separately by comparing $S$ to other subsets with similar columns.

Consider the first $s$ rows of $A(S)$.
Proposition \ref{factsaboutBC} tells us that, after suitable reordering, all but row $t$ has all its non-zero entries in the first $s-1$ columns. Therefore the last $r+1$ columns contain at most one non-zero entry which is in position $t$. As far as the order condition we are checking is concerned, it may be assumed that $n$ have entry $+1$ here and all others have entry zero. 
Now consider the subset $\overline{S}$ of $s+n-1$ vectors in $\mathbb{Z}^{s+n-1}$ where the first $s$ are the same as in $S$, except perhaps for the deletion of zero columns, and the last $n-1$ are given by $w_1 = e_s-e_{s+1}, \ldots w_{n-1} = e_{s+n-1} -e_{s+n}$. Note that the matrix of this subset has the same first $s$ rows as $A(S)$. The graph of $\overline{S}$ consists of the bad component $C$ and a chain of $n-1$ vertices of weight $-2$. 
The incidence matrix for this graph is given by $\overline{Q}=Q_1 \oplus Q_3$ where $Q_3$ is the incidence matrix for the chain of $-2$'s. This presents $\mathbb{Z}/{m^2 n} \oplus \mathbb{Z}/n$. The group $H(\overline{S})$ is of square root order so every element has order dividing $mn$. This shows that, for each of the columns with $s$ rows appearing in the upper part of $A(S)$, the vector given by multiplying by $mn$ is in the image of $Q_1$.

Now consider the last $r$ rows of $A(S)$. By Proposition \ref{factsaboutBC}, each of these has all the first $s$ entries zero.
Let $S'$ be the subset obtained from $S$ by replacing the bad component $C$ with the corresponding reduced component. This has a matrix with $r+1$ rows and columns and the last $r$ rows differ from those of $S$ only by the removal of the columns containing only zeros. The graph of $S'$ consists of the component of $S$ corresponding to $S \backslash C$ and an isolated vertex with weight $-n$, so $G(S')=\mathbb{Z}/{m^2 n} \oplus \mathbb{Z}/n$. Arguing as above, we see that the columns given by the last $r$ rows of each column of $A(S)$ gives an element of the image of $Q_2$ when multiplied by $mn$.

Thus, every column of $A(S)$ represents an element of order dividing $mn$ in $G(S)$, as claimed.
\end{proof}


The following technical result about subsets where every component is bad is also necessary. It will be convenient to introduce the following terminology. We call a subset $S$ of $m$ vectors in $\mathbb{D}^n$ square if $m=n$ and rectangular if $m=n+1$. 
Note that when $S$ is rectangular the matrix $Q_{\Gamma(S)}= -A(S)A(S)^t$ is singular.

\begin{Prop}\label{allBC}
If $S$ is a linear subset with $b(S)=c(S)=-I(S)$, $G(S)$ is not isomorphic to $H(S) \oplus H(S)$.
\end{Prop}
\begin{proof}
Since every component of $S$ is bad the group $G(S)$ is a direct sum of cyclic groups of the form $\mathbb{Z}/{m_i^2 n_i}$ ($1 \leq i \leq c(S)$). The condition that $c(S) + I(S) =0$ implies that \begin{equation}\label{all3}c(S) + \sum_{i=1}^{c(S)}{(n_i -4)} = -3c(S) + \sum_{i=1}^{c(S)}{n_i}=0.\end{equation}

By definition, each $n_i \geq 2$. By Proposition \ref{factsaboutBC}, there is a square linear subset $S'$ whose graph is given by $c(S)$ isolated vertices with weights $n_i$.
Suppose some $n_k=2$. The vector $v_k$ in $S'$ with $v_k \cdot v_k = -2$ can be linked to other vectors $v_j$. 
Suppose that for each of these vectors $n_j \geq 3$. Then, by deleting $v_k$ and the columns on which it is supported, we get a rectangular subset with graph given by isolated vertices with weight $n_i$ or $n_j-2$ with $n_j \geq 3$. This is not possible as the incidence matrix has non-zero determinant. A similar argument shows that $v_k$ must be linked to some $v_j$.

We now consider the possibility that some $n_j=2$. In this case $v_k$ can only be linked to the corresponding vector $v_j$ in $S'$ so it follows that there is a decomposition of $S$ as $T \cup T'$ where $T$ consists of the bad components built from $v_k$ and $v_j$. It now follows from Propositions \ref{irred} and \ref{twoBC} that $S$ is not a linear double subset.

We now turn to the case where each $n_i$ is at least 3. Condition (\ref{all3}) then implies that $n_i=3$ for each $i$.

Now, again, we can modify the subset $S$. For each bad component $C_i$, let $Q_i$ be the incidence matrix. Each $Q_i$ presents $\mathbb{Z}/{3 m_i^2 }$. The collection of rows of $A(S)$ corresponding to $Q_i$ is described by Proposition \ref{factsaboutBC}. In particular, we can obtain a subset $S'$ for the graph given by $C_i$ and a chain of $2$ vertices of weight $-2$ as in the proof of Proposition \ref{twoBC} by extracting the rows corresponding to $C_i$ from $A(S)$, modifying the central row with square $-3$ so that every entry is zero or one, deleting any zero columns and adding a pair of new rows given by $w_1 = e_t-e_{t+1}, w_{2} = e_{t+1} -e_{t+2}$.

Then $H(S')$ has order $3 m_i$. Let $M$ be the least common multiple of $\{3 m_i\}_{i=1}^{c(S)}$. It follows that $M H(S) =0$. 

Find a prime power $p^k$ and $i \in [1, c(S)]$ such that $p^k$ divides $m_i$ and $p^{k+1}$ does not divide any $m_j$.
There is an element in $G(S)$ of order $3p^{2k}$. However, it is clear $3 p^{2k}$ does not divide $M $ and this shows that there is no element of this order in $H(S) \oplus H(S)$.
\end{proof}


We say that a pair of components $C_1, C_2$ of a weighted graph are complementary if the manifolds $Y_i$ bounding the 4-manifolds produced by plumbing according to $C_i$ are such that $Y_1 \cong -Y_2$.

\begin{Prop}\label{complementarystuff}
Let $S$ be a linear subset such that \begin{equation}\label{oricond}c(S)+I(S) \leq 0 \text{ and }b(S) +I(S) <0.\end{equation}
If $S=\cup_iT_i$ where each $T_i$ is irreducible with $c(T_i) \geq 2$ and $b(T_i)=0$ then the graph of $S$ consists of pairs of complementary components.

In addition, for each $T_i$, there are generators $t,s$ for $G(T_i)$ such that $H(T_i)$ is generated by $t+s$ or $t-s$.
\end{Prop}

\begin{proof}
By \cite[Proof of Lemma 5.5]{lisca2} there is at least one $T_i$ which satisfies \eqref{oricond}. By \cite[Proposition 4.10]{lisca2} this must have $c(T_i)=2$.

It is shown in the proof of \cite[Lemma 5.4]{lisca2} that $T_i$ is as described in \cite[Lemma 4.7]{lisca2} and thus that plumbing on the graph of $T_i$ gives a manifold with boundary $L(p_i,q_i) \# L(p_i,p_i-q_i)$ for some $p_i,q_i$. This means that $c(T_i) + I(T_i)=0$ by \cite[Lemma 2.6]{lisca}. We can apply the same argument to each irreducible subset $T_i$ since it follows that \eqref{oricond} must hold for each.

When $c(T)=1$, a simple induction argument on the length $l$ of the chain shows that $G(T)=\mathbb{Z}^{l} / \operatorname{im} Q_{\Gamma(T)}$ is generated by $r=(1,0,\ldots ,0)^t$. Similarly, when $c(T)=2$ we easily find a pair of generators $t,s$ for $G(T) = \operatorname{coker} Q_1 \oplus Q_2$.

For every irreducible subset $T$ described by \cite[Lemma 4.7]{lisca2} either $t+s$ or $t-s$ is the first column of $A(T)$ and thus represents an element of $H(T)$. It follows from comparing the orders that this generates the group.
\end{proof}

We may now prove Theorem \ref{sumslens} by combining the above results with some results of \cite{lisca2}. Theorem \ref{non-ori} is proved in precisely the same way.
\begin{proof}[Proof of Theorems \ref{sumslens} and \ref{non-ori}]
Suppose $Y$ embeds in $S^4$, and so, consequently, does $-Y$.

There is a negative definite 4-manifold with boundary $Y$, for either orientation. Applying Corollary \ref{diag0} or \ref{diagnon} gives a linear double subset.

We may then choose an orientation. By \cite[Lemma 5.3]{lisca2}, we can assume that there is a linear double subset $S$ for which \eqref{oricond} holds.

Consider a decomposition of $S$ into irreducible components $S= \cup T_i$. By Corollary \ref{no1s} each has $c(T_i) \geq 2$.
Let $T$ be the union of all the $T_i$ which satisfy \eqref{oricond}. Each of these must then have $c(T_i)=2$ by \cite[Proposition 4.10]{lisca2}. Since $S$ is a double subset it follows from Propositions \ref{irred} and \ref{twoBC} that $T$ has no bad components and we may apply Proposition \ref{complementarystuff}.

Now consider $R=S\backslash T$.
This is possibly not irreducible and has $c(R) + I(R) \leq 0$ since the corresponding quantity is at most zero for $S$ and is equal to zero for each $T_i$. In order to have no irreducible component satisfy \eqref{oricond}, we must have $b(R)+I(R) \geq 0$. The fact that $b(R) \leq c(R)$ implies that $b(R)=c(R)=-I(R)$.

We require that $G(S) \cong H(S) \oplus H(S)$. Writing $S$ as the union of $R$ and $T$ gives $G(S) = G(T) \oplus G(R)$ and $H(S) = H(T) \oplus H(R)$.

It is clear that $G(T) \cong H(T) \oplus H(T)$. It follows that we must have $G(R) \cong H(R) \oplus H(R)$. 

However, this contradicts the result of Proposition \ref{allBC}. We conclude that $R$ is empty.
This proves that the Seifert invariants occur in weak complementary pairs as they are determined by $\Gamma$. Note that this graph does not distinguish between Seifert invariants of the form $$\frac{a}{b}=[a_1, \ldots ,a_n]^- \text{ and } \frac{a}{b'}=[a_n, \ldots ,a_1]^-.$$



We now use the second linear double subset given by Corollary \ref{diag0} or \ref{diagnon}. Each subset $S_k$ ($k=1,2$) is given by a union of irreducible subsets $T_{k,i}$, all of which satisfy (\ref{oricond}). Since the graphs of $S_1$ and $S_2$ are identical, we will just write $G(S)$ instead of $G(S_i)$. For each ratio $a/b$ let $T_k^{a/b}$ be the set of irreducible $T_{k,i}$ whose graph represents $L(a,b) \# L(a,a-b)$. The union of these subsets gives a summand $\left(\mathbb{Z}/a\right)^{2l}$ of $G(S)$. By Proposition \ref{complementarystuff}, this has generators $t_1,s_1, \ldots ,t_l,s_l$ and we can arrange that $H(T^{a/b}_1)$ is generated by $t_1+s_1, \ldots ,t_l +s_l$. There is a similar set of generators for $H(T^{a/b}_2)$ given by $\sigma(t_1)\pm \sigma(s_1), \ldots ,\sigma(t_l) \pm\sigma(s_l)$ where $\sigma$ is some permutation of $\{t_1,s_1, \ldots , t_n,s_n\}$.

When $a$ is even, $\frac{a}{2}\left(t_1+s_1+ \ldots + t_l+s_l\right)$ is an element of $H(T^{a/b}_1)$ and $H(T^{a/b}_2)$.
Theorem \ref{sumslens} now follows from Corollary \ref{diag0}. In the case of a non-orientable base orbifold, Corollary \ref{diagnon} implies that there can be at most one non-empty $T^{a/b}_i$ with $a$ even, completing the proof of Theorem \ref{non-ori}.

\end{proof}

\begin{Rmk}
The fact that any factor $L(p,q)$ in a connected sum of lens spaces embedding in $S^4$ has $p$ odd also follows from the linking form \cite{KK}.
\end{Rmk}



It is sometimes convenient to classify $S^1 \times S^2$ as a lens space since it also has a genus one Heegaard splitting.

\begin{Cor}
Let $L=\#L(p_i,q_i)$ with $p_i>q_i >0$ and suppose $L\#^n S^1 \times S^2$ embeds smoothly in $S^4$. Then $L$ also embeds smoothly.
\end{Cor}
\begin{proof}Replace the negative definite 4-manifold $X_L$ with boundary $L$ by $X_L \natural^n S^1 \times D^3$ and follow the proof of Theorem \ref{sumslens}.
\end{proof}

A similar approach gives a proof of Theorem \ref{e=0}.
\begin{proof}[Proof of Theorem \ref{e=0}]
Suppose $Y$ embeds smoothly in $S^4$.

By Corollary \ref{diag1} we have a rectangular subset $S$. The graph $\Gamma(S)$ is star-shaped and has a semi-definite incidence matrix.

Deleting the vector in $S$ corresponding to the central vertex of $\Gamma$ gives a new subset $S'$. This subset is linear and has the additional property that there is a vector $v$ which links once to a leaf of each component of the graph of $S'$ and not to any other vector.

Note that we may choose either orientation for $Y$ and so can assume that $S'$ satisfies condition \eqref{oricond}. We consider the irreducible components of $S'$. To apply Proposition \ref{complementarystuff} we need to show that every irreducible component $T$ has $c(T) \geq 2$ and $b(T)=0$.

Suppose that $c(T)=1$. Plumbing on the graph of $T$ gives the lens space $L(p,q)$ for some $p>q>0$. There is an extra vector $v$ such that $T \cup \{v\}$ is a rectangular subset and has a linear graph obtained from that of $T$ by adding a vertex onto one end, with weight $t$. Since the subset is rectangular, it follows that the determinant of the incidence matrix of this graph must be zero. However we can easily see that the graph is negative definite, so we conclude $c(T) \geq 2$. 

Now suppose that $T$ has a bad component $C$. By definition, this bad component can be built up from a linear chain of length three, with weights $-2, -n-1$ and $-2$ respectively. Let $C'$ be the component obtained from $C$ by deleting the vertex with weight $-n-1$. Suppose there is a new vertex $v$ which is only linked to one leaf of $C$ and consider the subset $T \cup \{v\}$. By Proposition \ref{factsaboutBC} each of the $r$ components of $C'$ is supported on $r$ columns of the matrix for this subset. We may then get a rectangular subset $T'$ by deleting the other columns and every row corresponding to $T \backslash C'$. The resulting graph has two components and is obtained from $C$ by adding a new vertex of weight $t$ to one end and deleting the vertex of weight $-n-1$. Similar to above, the incidence matrix of this graph is negative definite and so we conclude that $b(T)=0$.


It now follows from Proposition \ref{complementarystuff} that $Y$ has Seifert invariants occurring in (possibly weak) complementary pairs and, by \cite[Proposition 4.10]{lisca2}, that each irreducible $T_i$ has $c(T_i)=2$. Adding a new row $v_i$ to each $T_i$ gives a linear graph, which is negative definite graph whenever $v_i \cdot v_i < -1$. This shows that each $v_i \cdot v_i=-1$ and the result now follows from the description of the irreducible subsets in Proposition \ref{complementarystuff} and \cite[Lemma 4.7]{lisca2}.


\end{proof}

\section{Further obstructions from spin and spin$^{c}$ structures}\label{spin}
The methods described in the previous sections are more difficult to implement and give weaker obstructions in the case of Seifert manifolds with orientable base surfaces and $e \neq 0$. We therefore look for additional obstructions. Since we have chosen to primarily consider the case of double branched covers of pretzel links, we will focus on applications to that case when convenient.

If $Y$ is a closed, oriented 3-manifold it admits spin and spin$^c$ structures. Suppose $Y$ embeds smoothly in $S^4$ and splits it as $S^4 = U \cup_Y -V$. The 4-manifolds $U$ and $V$ must have both spin$^c$ and spin structures. There are obstructions to $Y$ embedding in $S^4$ which can be found by looking at the spin and spin$^c$ structures on $Y$ which can be extended over either $U$ or $V$.

\subsection{Spin$^c$ structures on rational homology spheres and the $d$ invariant}
If a manifold $Y$ admits spin$^c$ structures then the set of spin$^c$ is a $H^2(Y;\mathbb{Z})$-torsor. Suppose that $Y$ is a rational homology sphere which embeds smoothly in $S^4$. This gives a pair of rational homology balls $U,V$ such that $S^4=U \cup_Y -V$.
The spin$^c$ structures on $Y$ which arise as the restrictions of spin$^c$ structures on $U$ correspond to the image of the inclusion map $H^2(U) \to H^2(Y)$.
By Lemma \ref{basichom}, the inclusion maps induce an isomorphism $$H^2(Y;\mathbb{Z}) \cong H^2(U;\mathbb{Z}) \oplus H^2(V;\mathbb{Z}).$$

Since these two summands are isomorphic, there are $k^2$ spin$^c$ structures on $Y$. At least $2k-1$ of these spin$^c$ structures extend over a rational ball since $k$ extend over each of the rational balls $U$ and $V$ and only one -- the restriction of the unique spin$^c$ structure on $S^4$ -- extends over both pieces.

The correction term, or $d$ invariant, from Heegaard-Floer theory is a $\mathbb{Q}$-valued invariant of a rational homology 3-sphere with a spin$^c$ structure, first introduced in \cite{OSAbsolutegradings}. For our purposes, the relevant feature of this invariant is that whenever $(Y,\mathfrak{s})$ is a spin$^c$ 3-manifold and there is a rational ball $B$ bounding $Y$ with a spin$^c$ structure which restricts to $\mathfrak{s}$ on the boundary, then $d(Y,\mathfrak{s})=0$.

The $d$ invariant for a Seifert rational homology sphere can be determined using the associated star-shaped negative definite graph \cite{OSPlumbed} since it has at most one bad point.

It is described in \cite{GJ} how to relate this to the obstruction derived from Donaldson's theorem. This is used to obtain a stronger version of Theorem \ref{mainthm} in the case where $Y$ has the $\mathbb{Z}/2$-homology of $S^3$.
We may restate \cite[Theorem 3.6]{GJ} as follows.
\begin{Thm}
\label{diagd}
Let $Y$ be a 3-manifold with $H_*(Y;\mathbb{Z}/2) \cong H_*(S^3;\mathbb{Z}/2)$ which smoothly bounds a rational ball. Suppose that $Y$ bounds a negative definite plumbing $X$ with at most two bad points. The vertices of this plumbing give a basis for $H_2(X)$ and we may then identify $H^2(Y;\mathbb{Z})$ with $\operatorname{coker}Q_X$.

Then there is a matrix $A$ such that $Q_X=-AA^t$ and every class of $\frac{\operatorname{im} A}{\operatorname{im} Q_X}$ contains a characteristic representative of the form $Ax$ for some $x \in\{\pm 1\}^{n}$. 
\end{Thm}


\subsection{Spin structures and the $\overline{\mu}$ invariant}

If a manifold $Y$ admits a spin structure then the set of spin structures on $Y$ is a torsor for $H^1(Y;\mathbb{Z}/2)$. As with spin$^c$ structures, if $Y$ is a 3-manifold which embeds in $S^4$ there is an isomorphism induced by inclusion maps $$H^1(Y;\mathbb{Z}/2) \cong H^1(U;\mathbb{Z}/2) \oplus H^1(V;\mathbb{Z}/2).$$ 

\begin{Lem}\label{linkcomplem}If $Y$ is the double branched cover of a $k$-component link $L$ then it has $2^{k-1}$ spin structures. If $Y$ embeds smoothly in $S^4$ then $b_1(Y)$ is even if and only if $k$ is odd. In particular, when $L$ is a pretzel link $b_1(Y)$ is zero when $k$ is odd and one when $k$ is even.

\end{Lem}
\begin{proof}
A statement analogous to Lemma \ref{basichom} (2) holds for $\mathbb{Z}/2$-coefficient first cohomology, also due to Alexander duality, and shows that the number of spin structures on a 3-manifold $Y$ embedding in $S^4$ is $2^{b_1(Y)} l^2$ for some integer $l$. This is a square precisely when the first Betti number is even. By \cite{turaev} there is a correspondence between quasiorientations of a link and spin structures on the double branched cover. For a $k$-component link there are $2^{k -1}$ spin structures on the double branched cover and this is a square precisely when $k$ is odd.

When $L$ is a pretzel link, $Y$ is a Seifert manifold with base $S^2$ and it follows from, for example, \cite[Theorem 3.1]{hillman}, that $b_1(Y) \leq 1$.


\end{proof}

Let $Y$ be given as the boundary of a 2-handlebody $X$ represented by a framed link in $S^3$.
\begin{Defn}\label{defcharsub}
A sublink $L'$ of a framed link $L$ is called characteristic if for every component $K$ of $L$ the total linking number $\operatorname{lk} (K, L')$ is congruent modulo 2 to 
the framing on $K$.
\end{Defn}

Spin structures on $Y$ correspond bijectively to characteristic sublinks of the diagram for $X$ (see \cite[Proposition 5.7.11]{GS} for example). The characteristic sublink of a spin structure $\mathfrak{s}$ represents an obstruction to extending $\mathfrak{s}$ over the 2-handlebody. If the characteristic sublink is empty, the 2-handlebody has a unique spin structure which restricts to $\mathfrak{s}$ on the boundary.

Kaplan \cite{Kaplan} gives an algorithm which produces a spin 2-handlebody extending a given spin structure on any 3-manifold. The algorithm uses handle-slides and blow-ups to remove the characteristic sublink. We briefly recall the effects of these moves on characteristic sublinks. (See \cite[Section 5.7]{GS} for a more detailed discussion.) If we slide one component of a characteristic sublink over another the characteristic sublink in the new diagram simply contains the new curve, and so has one fewer component. The new curve added in a blow-up is included in a characteristic sublink if and only if it has an even linking number with the sublink. If we blow down a component in a characteristic sublink then the corresponding characteristic sublink in the resulting diagram consists of the other curves in the original.


Suppose that $X$ is given by plumbing on a tree $\Gamma$. The spin structures on the boundary of $X$ now correspond bijectively to subsets of the vertex set of $\Gamma$ which are characteristic for the incidence matrix of $\Gamma$. Such sets, or equivalently the classes they represent in $H_2(X;\mathbb{Z}/2)$, are called (homology) Wu sets and are always isolated.

\begin{Defn}Let $X$ be a plumbing according to a weighted tree. The Neumann-Siebenmann $\overline{\mu}$ invariant of $Y=\partial X$ with spin structure $\mathfrak{s}$ corresponding to a Wu set $w$ is defined as $\overline{\mu}(Y,\mathfrak{s}) = \sigma(X) - w \cdot w$.
\end{Defn}
It is shown in \cite{neumann} that this only depends on $(Y,\mathfrak{s})$ and not on the 4-manifold $X$ used in the construction. It is apparent that this is a lift of the Rochlin invariant.

We consider the $\overline{\mu}-$invariant for  Seifert manifolds with spin structures which extend over 4-manifolds with simple rational homology. The key result is Furuta's 10/8 theorem. For Seifert rational homology spheres the $\overline{\mu}$-invariant is known to be a spin rational homology cobordism invariant \cite{Ue} (see also \cite{savelievmubarzhs} for integer homology spheres), which is proved using a V-manifold version of the ${10}/{8}$ theorem \cite{FF}.

Here, we will give an alternative argument which is applicable for the cases we are most interested in, including some with positive first Betti number. Our approach is similar to \cite{bohr-lee}, which derives a knot sliceness obstruction from Furuta's theorem.

\begin{Thm}[Furuta \cite{furuta}]
Let $W$ be a closed, spin, smooth 4-manifold with an indefinite intersection form. Then $$4b_2(W) \geq 5 |\sigma(W)| +8.$$
\end{Thm}

Note that, by Donaldson's diagonalisation theorem, a closed, smooth, spin manifold $W$ can have a definite intersection form only if $b_2(W)=0$.
\begin{Lem}\label{spinnumbers}Let $(Y,\mathfrak{s})$ be a 3-manifold with a chosen spin structure. Suppose that $(X,\mathfrak{s}_X)$ is a spin 2-handlebody and $(V,\mathfrak{s}_V)$ is a spin manifold with $b_3(V)=0$ such that $\partial (X,\mathfrak{s}_X) = \partial (V,\mathfrak{s}_V) = (Y,\mathfrak{s})$.

Then $W=X \cup_Y -V$ is spin with signature $\sigma(W) = \sigma(X) + \sigma(V)$ and $b_2(W) = b_2(X) + \chi(V) -1$.
\end{Lem}

\begin{proof}The fact that $W$ is spin follows since the spin structures on $X$ and $V$ agree on the boundary.

It is easy to see that $\chi(W) = \chi(X) + \chi(V)=1+b_2(X) +\chi(V)$.

Since $H_1(W,X;\mathbb{Q}) \cong H_1(V,Y;\mathbb{Q})=0$ it follows from the exact sequence for the pair $(W,X)$ that $b_1(W)=0$. The result now follows from the calculation of the Euler characteristic and Novikov additivity.
\end{proof}

To get an obstruction to a 3-manifold $Y$ with $b_1(Y) \leq 1$ embedding in $S^4$, we consider the case where $V$ is one of the spin pieces obtained from the splitting induced by an embedding.

\begin{Cor}\label{10/8cons}Let $(Y,\mathfrak{s})$ be a spin 3-manifold and let $(V,\mathfrak{s}_V)$ be a spin manifold and $(X,\mathfrak{s}_X)$ be a spin 2-handlebody with common boundary $(Y,\mathfrak{s})$.
\begin{enumerate} \item If $V$ is a rational ball then either $X=D^4$ or $$4 b_2(X) \geq 5|\sigma(X)| + 8;$$
\item If $H_*(V;\mathbb{Q}) = H_*(S^1 ;\mathbb{Q})$ then either $b_2(X)=1$ or $$4 b_2(X) \geq 5|\sigma(X)| + 12;$$
\item If $H_*(V;\mathbb{Q}) = H_*(S^2 ;\mathbb{Q})$ then $$4 b_2(X) \geq 5|\sigma(X)+ \sigma(V)| + 4.$$
\end{enumerate}
\end{Cor}
\begin{proof}
We apply Furuta's theorem and Lemma \ref{spinnumbers} to the closed, spin manifold $W=X \cup_Y -V$.
\end{proof}

We now construct spin 4-manifolds bounding double branched covers of pretzel links.
\begin{Prop}\label{spins}
Let $Y$ be the double branched cover of a $3$ or $4$-stranded pretzel link and let $\mathfrak{s}$ be a spin structure on $Y$. Then there is a spin 2-handlebody $(X,\mathfrak{s}_X)$ with spin boundary $(Y,\mathfrak{s})$ with signature $\sigma(X)=\overline{\mu}(Y,\mathfrak{s})$ and $0 \leq b_2(X) - |\sigma(X)| \leq 4$.
\end{Prop}
\begin{proof}
Let $X'$ be one of the 2-handlebodies shown in Figure \ref{spinpic}. The boundary is the same as the 2-handlebodies pictured in Figure \ref{pretzelc} -- we can slide over the component with framing $a$ and then exchange the $0$ framed unknot for a $1$-handle and cancel. 

\begin{figure}[htbp] 
\begin{center}
\ifpic
\def\svgwidth{10cm}
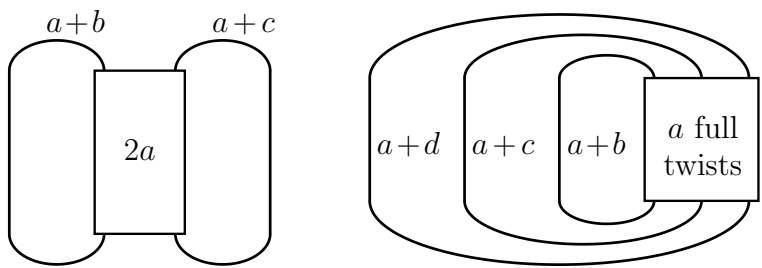
\else \vskip 5cm \fi
\begin{narrow}{0.3in}{0.3in}
\caption{
\bf{$X'$ for $n=3,4$.}}
\label{spinpic}
\end{narrow}
\end{center}
\end{figure}


Every sublink of $X'$ is potentially characteristic, depending on $a,b,c$ and $d$. For each spin structure $\mathfrak{s}$ on $\partial X'$ we can arrange by handleslides that the characteristic sublink is an unknot as follows. If the sublink containing the two components of framings $a+b$ and $a+c$ is characteristic we can slide one over the other to get a single unknotted component with framing $b+c$.
In the 4-strand case, there may be a characteristic sublink with three components. If we perform the handle slide above, the resulting picture has a characteristic unlink. It is then obvious that we can slide one component over the other.

This gives a new diagram for $X'$ in which the characteristic sublink is an unknot with framing $n$. The $\overline{\mu}$ invariant of $(Y,\mathfrak{s})$ is $\sigma(X') - n$. This can easily be verified using the above description of the handle moves needed to convert the plumbing tree to $X'$.

By reversing the orientation of $X'$ if necessary, we may assume $\sigma(X') \geq 0$. Note that since $X'$ has only a small number of handles, $\sigma(X')\leq 3$. 
We now consider various cases depending on the sign of $n$.

If $n=0$ then we can remove the characteristic sublink by blowing up a $+1$ meridian of it and then blowing down the resulting $+1$ framed curve. This gives an $X$ with signature $\overline{\mu}(Y,\mathfrak{s})$ and $b_2(X) = b_2(X') \leq 3$.

If $n<0$ then the characteristic sublink can be removed by blowing up $|n|-1$ meridians with framing $+1$ and then blowing down the resulting $-1$ curve. This produces a spin manifold $X$ with $\sigma(X) = \sigma(X') -n$ and $b_2(X)=b_2(X') + |n| -2$.

By the assumptions on $X'$ and $n$, $\sigma(X) >0$ so \begin{align*}b_2(X) - |\sigma(X)| &= b_2(X') - n -2 -\sigma(X') +n\\ &= b_2(X') - \sigma(X')-2.\end{align*}
This is at most $1$.

If $n>0$, the characteristic sublink can be removed by blowing up a $-1$-framed meridian of the characteristic link $n-1$ times and blowing down a $+1$ curve. This gives a spin manifold $X$ with $\sigma(X) = \sigma(X') - n$ and $b_2(X) = b_2(X') +n-2$.

If $\sigma(X)\geq 0$ then necessarily $n \leq 3$. Then $b_2(X) \leq b_2(X') +1 \leq 4$. Alternatively, if $\sigma(X) <0$ then \begin{align*}b_2(X) - |\sigma(X)| &= b_2(X')+n-2 + \sigma(X') -n \\ &= b_2(X') + \sigma(X') -2.\end{align*}This is, again, at most $4$.

\end{proof}

We can apply Corollary \ref{10/8cons} to produce the following conclusions.
\begin{Cor}\label{mubar}
Let $Y$ be the double branched cover of a 3 or 4 stranded pretzel link with $k$ components. If $Y$ embeds in $S^4$ then the Neumann-Siebenmann $\overline{\mu}$ invariant vanishes for at least $2^{\frac{k+1}{2}} - 1$ spin structures on $Y$ if $k$ is odd and at least $3(2^{\frac{k-2}{2}}) -1$ if $k$ is even.

\end{Cor}
\begin{proof}
Since $Y$ embeds smoothly in $S^4$ we can write $S^4=U \cup_Y -V$. Since $Y$ is the double branched cover of a pretzel link $b_1(Y) \leq 1$. Lemma \ref{basichom} implies that for both $U$ and $V$ the sum of the first and second Betti number is at most one.

For every spin structure $\mathfrak{s}$ extending over either $U$ or $V$ we apply Corollary \ref{10/8cons} to the 2-handlebody $X$ given by Proposition \ref{spins}. This shows that $$4b_2(X) \geq 5 |\overline{\mu}(Y,\mathfrak{s})| + 4.$$ Since $b_2(X) \leq |\overline{\mu}(Y,\mathfrak{s})|+4$ we see that $|\overline{\mu}(Y,\mathfrak{s})| \leq 12$.

Since $U$ and $V$ both have signature zero, it follows from Rochlin's theorem that the $\overline{\mu}$ invariant vanishes for every spin structure extending over $U$ or $V$.

The proof of Lemma \ref{linkcomplem} shows that the total number of spin structures on $Y$ is $2^{b_1(Y)}l^2$ where $2^{b_1(Y)}l$ spin structures extend over $U$ and $l$ extend over $V$. Exactly one extends over both to give the unique spin structure on $S^4$. The result now follows since $b_1(Y)$ is determined by the parity of $k$.
\end{proof}





\section{Double branched covers of pretzel links}\label{pretzelcovers}
The proof of Theorem \ref{pretzelcover34} will use a combination of the obstructions from Sections \ref{diagonalisation} and \ref{spin}.
Recall that $Y(a,b,c)$ and $Y(a,b,c,d)$ denote the double branched covers of $P(a,b,c)$ and $P(a,b,c,d)$ respectively. All of the positive embedding results are demonstrated in Section \ref{construction}. This section will complete the proof by outlining the necessary obstructions.

It will be convenient to use Corollary \ref{mubar} as our principal obstruction. Accordingly, we consider cases with different numbers of spin structures separately. By Lemma \ref{linkcomplem} this is equivalent to splitting up into cases according to the number of link components.

We first consider the cases with first Betti number one. Note that these fall under the hypothesis of Theorem \ref{e=0} and so every example of this type which embeds smoothly in $S^4$ is of the form $Y(a,-a,b,-b)$. 

\begin{Prop}Suppose that $a>b>0$ and $a,b$ are both even. Then $Y=Y(a,-a,b,-b)$ does not embed smoothly in $S^4$.
\end{Prop}
\begin{proof}An easy calculation using the plumbing in Figure \ref{pretzelc} shows that $Y$ has eight spin structures and that only four have vanishing $\overline{\mu}$ invariant. The others are $\pm(a\pm b)$. Corollary \ref{mubar} shows that these do not embed smoothly in $S^4$.
\end{proof}
\begin{Rmk}This demonstrates that Theorem \ref{e=0} does not give a complete obstruction.
\end{Rmk}

We now consider the double branched covers of links with odd numbers of components.
\subsection{Double branched covers of knots}

Due to interest in the question of knot sliceness, there are previous results we may appeal to. In particular, for pretzel knots, the possible form of subsets appearing in Theorem \ref{mainthm} have been computed \cite{GJ} \cite{Lec}. The $\overline{\mu}$ invariant is useful as an obstruction to a knot being slice since any 4-manifold with the $\mathbb{Z}/2$ homology of $D^4$ is necessarily spin. Indeed, for Montesinos knots the $\overline{\mu}$ invariant of the double branched cover agrees with the knot signature \cite{savelievmubar} and the resulting obstruction is incorporated into the results of \cite{GJ} and \cite{Lec}.

To begin with, we consider the double branched covers of 3-stranded pretzel knots. There are two cases to consider. We assume that $Y(a,b,c)$ has a positive generalised Euler characteristic and consider how many of $a,b,c$ are positive.

\begin{Prop}
Let $Y(a,b,c)$ be the double branched cover of a knot with $a,b>1$ and $e(Y)>0$. Then if $Y$ embeds smoothly in $S^4$, $c<0$ and $Y$ is diffeomorphic to $Y(a,-a,a)$.
\end{Prop}
\begin{proof}
First, note that if $c$ is also positive it is impossible to have a vanishing $\overline{\mu}$ invariant.

The case where $Y(a,b,c)$ is a $\mathbb{Z}/2$ homology sphere with $a,b>0$ and $c<0$ is by Greene and Jabuka \cite{GJ}. Note that while they only explicitly consider the case where $a,b,c$ are odd, this is only important in their calculation of the knot signature and has no effect on their arguments using Donaldson's diagonalisation theorem or the $d$ invariant. Their Proposition 3.1 determines the possible subsets in this case to be uniquely determined up to a choice of a parameter $\lambda$ such that $-c = \lambda^2 a + (\lambda +1)^2 b$.

Greene and Jabuka use the $d$ invariant, in the way described in Theorem \ref{diagd}, to show that this $\lambda$ must be either $-1$ or $0$ if $Y$ is the boundary of a rational ball. This shows that $-c=a$ or $-c=b$. Note that $Y(a,b,-a)$ has first homology of order $a^2$ so it can only be a homology sphere if it is $S^3$. Otherwise, we may apply Corollary \ref{diag0} to show that there must be a second subset. This means that both $\lambda=0$ and $\lambda=-1$ must be valid. It follows that $a=b=-c$.
\end{proof}

Now we consider the case where $Y(a,b,c)$ has just $a$ positive.

\begin{Prop}
Let $Y=Y(a,b,c)$ be the double branched cover of a knot with $a>1$, $b,c<-1$ and $e(Y)>0$.
Then if $Y$ embeds smoothly in $S^4$ then it is a homology sphere of the form $Y(2\lambda-1,-2\lambda-1,-2\lambda^2)$.
\end{Prop}
\begin{proof}
By systematically checking other possibilities it is not difficult to verify that, in order to have a unique $\overline{\mu}$ invariant of zero, we must have $c$ even and $b=-a-2$ odd, up to relabeling $b$ and $c$.

Consider the 4-manifold $X'$ with boundary $Y$ shown in Figure \ref{spin3thing}, where $2a$ refers to the number of crossings.

\begin{figure}[htbp] 
\begin{center}
\ifpic
\def\svgwidth{4cm}
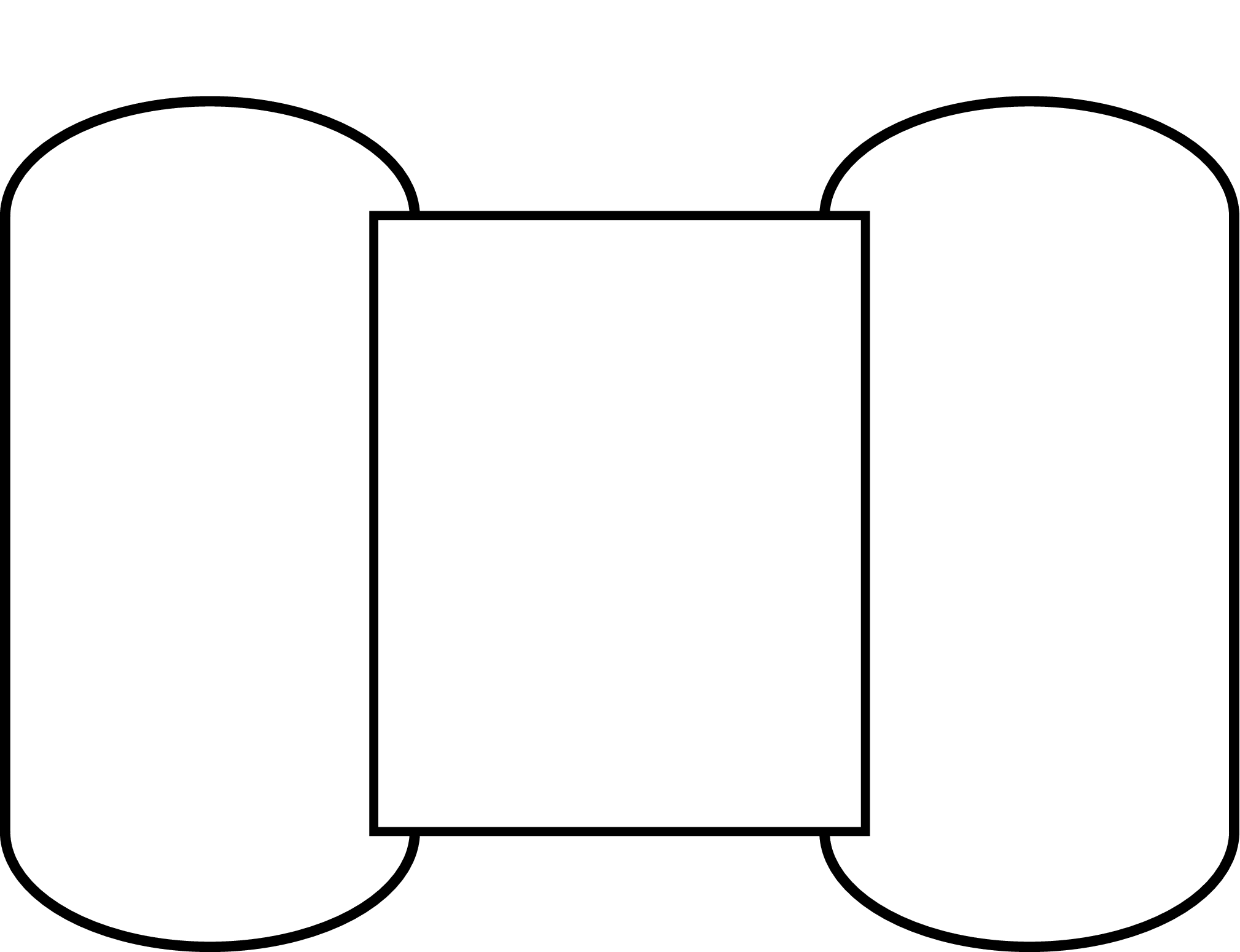
\else \vskip 5cm \fi
\begin{narrow}{0.3in}{0.3in}
\caption{
\bf{$X'$.}}
\label{spin3thing}
\end{narrow}
\end{center}
\end{figure}

The intersection form of $X'$ has determinant $ab+ac+bc >0$ so it is definite. Since $a+b=-2$ it must be negative definite.

The possible subsets we get from applying Corollary \ref{diag0} give matrices of the form $$A_i = \begin{pmatrix} 1 & 1 \\ \rho & \lambda \end{pmatrix}.$$

Up to a change of basis of the columns space this is unique. 
This means that there is only one subset. This provides an obstruction unless $Y$ is a homology sphere as noted in Remark \ref{zhsf}.
In this case we require that $\det A = \lambda - \rho =\pm 1$. Up to relabelling we can assume $\rho = \lambda -1$.

It then follows that $a=2\lambda -1$, $b=-2\lambda-1$ and $c=-2\lambda^2$.
\end{proof}

We now consider the double branched covers of 4-strand pretzel knots.
\begin{Prop}
Let $Y=Y(a,b,c,d)$ with $a,b,c,d \in \mathbb{Z}\backslash \{0\}$ be a $\mathbb{Z}/2$ homology sphere which embeds smoothly in $S^4$.
Then $Y$ embeds smoothly in $S^4$ if and only if it is diffeomorphic to $Y(a\pm 1,-a,a,-a)$.
\end{Prop}
\begin{proof}
Suppose $Y$ embeds smoothly in $S^4$. We consider the condition imposed by Corollary \ref{diag0}.

In \cite[Lemma V.6]{Lec} the subset obtained by viewing the standard negative definite plumbing as a submanifold of a closed definite manifold is uniquely determined and, in conjunction with the $\overline{\mu}$ invariant, is used to show that if $Y$ bounds a rational ball it is either $Y(-a,-b-1,a,b)$ with $a,b <-1$ or, if $a=1$, has the form $Y(1,-2,b,-b) \cong Y(2,b,-b)$. The latter is considered above and does not embed smoothly in $S^4$.

The subset $S$ for $Y(-a,-b-1,a,b)$ is described explicitly by \cite[Figure V.5]{Lec} and is obtained by adding a new column with a single non-zero entry to the matrix for the essentially unique rectangular subset for $Y(a,-a,b,-b)$.

On inspection we see that in order to get a second subset, which differs as specified by Corollary \ref{diag0}, we must have $a=b$.
\end{proof}

\subsection{Double branched covers of 3-component links}

Finally we consider double branched covers of pretzel links with three components. By Lemma \ref{linkcomplem} the double branched covers have four spin structures and, if they embed in $S^4$, are rational homology spheres.

We first consider the following special case, where Corollary \ref{mubar} is not sufficient.
\begin{Prop}\label{shitmess}Let $a$ be odd and $b$ even. If $Y(a,b,b,b)$ embeds smoothly in $S^4$ and has $e(Y)>0$ then it is diffeomorphic to $Y(2,-2,2)$.
\end{Prop}
\begin{proof}
In order to find a subset, $b$ must be negative or $2$. We can see this by a simple extension of the proof of \cite[Lemma V.5]{Lec}, where we drop the assumption that $Y$ is a $\mathbb{Z}/2$ homology sphere -- we attempt to construct a subset and compare the number of columns required to the number of vertices in the graph.
The $\overline{\mu}$ invariants for $Y(a,2,2,2)$ can easily be calculated and three are $\operatorname{sign}a -a$. The manifold $Y(-1,2,2,2) \cong Y(2,-2,2)$ embeds in $S^4$ but $Y(1,2,2,2)$ does not as it has first homology of non-square order 20.

In the case where $b<0$, the generalised Euler invariant implies that $a>0$. Calculating the $\overline{\mu}$ invariants shows that $a=-b-3$. The condition that $a>0$ means that $b<-3$.

We can now express the generalised Euler characteristic as $$\frac{1}{-b-3}+\frac{3}{b} = \frac{-2b-9}{-b^2-3b} >0.$$
Since the denominator in this fraction is $ab<0$, this shows that $b \geq -4$.

To show that $Y(1,-4,-4,-4)$ does not embed in $S^4$, we use Corollary \ref{diag0}.

For the standard definite plumbing, a simple computation shows that the matrix $A(S)$ is uniquely determined up to reordering or changing the signs of the columns.
\end{proof}

Finally, we consider the last remaining case needed to prove Theorem \ref{pretzelcover34}
\begin{Prop}
Let $Y$ be of the form $Y(a,b,c)$ or $Y(a,b,c,d)$ where $a,b,c,d \in \mathbb{Z}\backslash \{0\}$. Suppose that $Y$ has four spin structures. Then $Y$ embeds smoothly in $S^4$ if and only if it is diffeomorphic to either $Y(a,-a,a)$ or $Y(a \pm 1,-a,a,-a)$.

\end{Prop}

\begin{proof}
We first consider $Y=Y(a,b,c)$. This has four spin structures only when $a,b$ and $c$ are even. Let $\tau$ be the signature of the 4-manifold given by the first diagram in Figure \ref{pretzelc}. The four $\overline{\mu}$ invariants of $Y$ are $\tau, \tau-a-b, \tau-a-c$ and $\tau-b-c$. Three are zero which implies that either $\tau=0$ and, up to reordering, $a=b=-c$ or $a=b=c$. In the latter case $\tau= \pm 2$ and so $a=b=c=\pm1$. This does not embed in $S^4$ as it is either the lens space $L(3,1)$ or $L(3,2)$.

Next, consider $Y=Y(a,b,c,d)$. This has four spin structures if exactly one is odd, which can be assumed to be $a$. Define $\tau$, similar to the above, using the second picture in Figure \ref{pretzelc}.
The $\overline{\mu}$ invariants are $\tau-a-b, \tau-a-c, \tau-a-d$ and $\tau-a-b-c-d$.

We again require that three are zero. If the the last of these is not, we may apply Proposition \ref{shitmess}. Otherwise, up to relabeling, $b=c=-d$.
It follows easily, by considering the value of $\tau$ for either sign of $b$, that $a=-b\pm 1$.
\end{proof}


\bibliographystyle{amsplain}
\bibliography{ref}

\providecommand{\bysame}{\leavevmode\hbox to3em{\hrulefill}\thinspace}
\providecommand{\MR}{\relax\ifhmode\unskip\space\fi MR }
\providecommand{\MRhref}[2]{%
  \href{http://www.ams.org/mathscinet-getitem?mr=#1}{#2}
}
\providecommand{\href}[2]{#2}
\begin{thebibliography}{10}

\bibitem{bohr-lee}
C~Bohr and R~Lee, \emph{Homology cobordism and classical knot invariants},
  arXiv:math/0104042v1 (2001).

\bibitem{Budney}
R.~Budney, \emph{Embeddings of 3-manifolds in ${S}^4$ from the point of view of
  the 11-tetrahedron census}, arXiv:0810.2346v4 (2010).

\bibitem{CH}
J.S. Crisp and J.A. Hillman, \emph{Embedding {S}eifert fibred 3-manifolds and
  {S}ol$^3$-manifolds in 4-space}, Proc. London Math Soc. \textbf{3} (1998),
  685--710.

\bibitem{donaldson}
S.K. Donaldson, \emph{The orientation of {Y}ang-{M}ills moduli spaces and
  4-manifold topology}, J. Differential Geom. \textbf{26} (1987), 397--428.

\bibitem{fintushel-sternssf}
R.~Fintushel and R.~Stern, \emph{Rational homlogy cobordisms of spherical space
  forms}, Topology \textbf{26} (1987), no.~3, 385--393.

\bibitem{FF}
Y.~Fukumoto and M.~Furuta, \emph{Homology 3-spheres bounding acyclic
  4-manifolds}, Math. Res. Lett. \textbf{7} (2000), 757--766.

\bibitem{furuta}
M~Furuta, \emph{Monopole equation and the 11/8 conjecture}, Math. Res. Lett.
  \textbf{8} (2001), 279--291.

\bibitem{GL}
P.~M. Glimer and C.~Livingston, \emph{On embedding 3-manifolds in 4-space},
  Topology \textbf{22} (1983), no.~3, 241--252.

\bibitem{GS}
R.~E. Gompf and A.~I. Stipsicz, \emph{4-manifolds and {K}irby {C}alculus},
  Graduate Studies in Mathematics, vol.~20, Amer. Math. Soc., 1999.

\bibitem{GJ}
J.~Greene and S.~Jabuka, \emph{The slice-ribbon conjecture for 3-stranded
  pretzel knots}, Amer. J. Math. \textbf{133} (2011), 555--580.

\bibitem{hillman}
J.~A. Hillman, \emph{Embedding 3-manifolds with circle actions}, Proc. Amer.
  Math. Soc. \textbf{137} (2009), no.~12, 4287--4294.

\bibitem{hosokawa}
F.~Hosokawa, \emph{On trivial 2-spheres in 4-space}, Quart. J. Math. Oxford
  Ser. \textbf{19} (1968), no.~2, 249--256.

\bibitem{Kaplan}
S.~J. Kaplan, \emph{Constructing framed 4-manifolds with given almost framed
  boundaries}, Trans. Amer. Math. Soc. \textbf{254} (1979), 237--263.

\bibitem{KK}
A.~Kawauchi and S.~Kojima, \emph{Algebraic classification of linking pairings
  on 3-manifolds}, Math. Ann. \textbf{253} (1980), no.~1, 29--42.

\bibitem{Lec}
A.~G. Lecuona, \emph{On the slice-ribbon conjecture for {M}ontesinos knots},
  Ph.D. thesis, Universita di Pisa, 2009/2010.

\bibitem{lisca}
P.~Lisca, \emph{Lens spaces, rational balls and the ribbon conjecture},
  Geometry \& Topology \textbf{11} (2007), 429--472.

\bibitem{lisca2}
\bysame, \emph{Sums of lens spaces bounding rational balls}, Algebraic \&
  Geometry Topology \textbf{7} (2007), 2141--2164.

\bibitem{neumann}
W.~Neumann, \emph{An invariant of plumbed homology spheres}, Topology
  Symposium, Siegen 1979, Lect. Notes in Math., vol. 788, Springer, Berlin,
  1980, pp.~125--144.

\bibitem{NR}
W.~Neumann and F.~Raymond, \emph{Seifert manifolds, plumbing, $\mu$-invariant
  and orientation reversing maps}, Lecture Notes in Math. \textbf{664} (1978),
  163--196.

\bibitem{OSAbsolutegradings}
P.~Oszv\'ath and Z.~Szab\'o, \emph{Absolutely graded {F}loer homlogies and
  intersection forms for four-manifolds with boundary}, Adv. Math. \textbf{173}
  (2003), no.~2, 179--261.

\bibitem{OSPlumbed}
\bysame, \emph{On the {F}loer homology of plumbed three-manifolds}, Geometry \&
  Topology \textbf{7} (2003), 185--224.

\bibitem{rolfsen}
D.~Rolfsen, \emph{Knots and links}, AMS Chelsea, 2003.

\bibitem{savelievmubar}
N.~Saveliev, \emph{A surgery formula for the $\overline{\mu}$ invariant},
  Topology Appl. \textbf{106} (2000), 91--102.

\bibitem{savelievmubarzhs}
\bysame, \emph{Fukumoto-{F}uruta invariants of plumbed homology 3-spheres},
  Pacific J. Math. \textbf{205} (2002), no.~2, 465--490.

\bibitem{Scharlemann}
M.~Scharlemann, \emph{Smooth spheres in $\mathbb{R}^4$ with four critical
  points are standard}, Invent. Math. \textbf{79} (1985), 125--141.

\bibitem{turaev}
V.G. Turaev, \emph{Classification of oriented {M}ontesinos links via spin
  structures}, Lecture Notes in Math. \textbf{1346} (1988), 271--289.

\bibitem{Ue}
M.~Ue, \emph{The {N}eumann-{S}iebenmann invariant and {S}eifert surgery}, Math.
  Z. \textbf{250} (2005), 475--493.

\bibitem{Zeeman}
E.C. Zeeman, \emph{Twisting spun knots}, Trans. Amer. Math. Soc. \textbf{115}
  (1965), 471--495.

\end{thebibliography}

\end{document}